\documentclass[11pt,a4paper]{article}
\usepackage[top=2.5cm,bottom=2.5cm,left=2cm,right=2cm]{geometry}
\usepackage{mathrsfs}
\usepackage{pstricks}
\usepackage{algorithm}
\usepackage{algorithmic} 
\usepackage{amsmath,amssymb,graphicx,amsthm,tabularx,array}
\usepackage{hyperref}
\usepackage{makecell}
\usepackage{bm}
\usepackage{float} 
\usepackage{booktabs}  

\theoremstyle{plain}
\newtheorem{theorem}{Theorem}[section]

\newtheorem{coro}[theorem]{Corollary}
\newtheorem{lem}[theorem]{Lemma}

\newtheorem{conj}[theorem]{Conjecture}
\newtheorem{problem}[theorem]{Problem}

\theoremstyle{definition}
\newtheorem{example}{Example}[section]
\newtheorem{definition}[theorem]{Definition}

\newtheorem{other}{}

\title{Some Tur\'{a}n-type results for {the} signless Laplacian spectral radius}

\author{
Jian Zheng\thanks{
School of Mathematics and Statistics, Jiangxi Normal University, Jiangxi, China. 
Email: \url{zhengj@jxnu.edu.cn}. }  
 \and 
 Yongtao Li\thanks{Corresponding author, 
 Yau Mathematical Sciences Center, Tsinghua University, Beijing, China. Email: \url{ytli0921@hnu.edu.cn}.}
\and 
Yi-Zheng Fan\thanks{
School of Mathematical Sciences, Anhui University, Hefei, China. Email: \url{fanyz@ahu.edu.cn}. 
} 
}

\date{\today}

\begin{document}
\maketitle

\vspace{-1cm}

\begin{abstract}
 Half a century ago, Bollob\'{a}s and Erd\H{o}s [Bull. London Math. Soc. 5 (1973)] proved that every $n$-vertex graph $G$ with $e(G)\ge (1- \frac{1}{k} + \varepsilon )\frac{n^2}{2}$ edges contains a blowup $K_{k+1}[t]$ with $t=\Omega_{k,\varepsilon}(\log n)$. 
A well-known theorem of Nikiforov [Combin. Probab. Comput. 18 (3) (2009)] asserts 
that if $G$ is an $n$-vertex graph with adjacency spectral radius 
$\lambda (G)\ge (1- \frac{1}{k} + \varepsilon)n$, then $G$ contains a blowup $K_{k+1}[t]$ with $t=\Omega_{k,\varepsilon}(\log n)$. This gives a spectral extension of the Bollob\'{a}s--Erd\H{o}s theorem. 
 In this paper, we systematically explore variants of Nikiforov's result in terms of the signless Laplacian spectral radius, extending the supersaturation, blowup of cliques and the stability results. 

\begin{itemize}
    \item 
We prove that if $k\ge 2,\varepsilon>0$ and $G$ is an $n$-vertex graph with $q(G)\ge (1- \frac{1}{k} + \varepsilon)2n$, then $G$ contains $\Omega_{k,\varepsilon}(n^{k+1})$ copies of $K_{k+1}$. This extends the adjacency  spectral supersaturation of Bollob\'{a}s--Nikiforov. Moreover, such a graph $G$ contains a blowup $K_{k+1}[t]$ of size $t=\Omega_{k,\varepsilon}(\log n)$. {This generalizes the result of Nikiforov as well as the result of Zheng--Li--Su.}  

    \item 
    We show that if $k\ge 3$ and $G$ is an $n$-vertex graph with $q(G)> (1- \frac{1}{k})2n$,  then 
    $G$ contains a color-critical graph $K_{k}^+[t]$ with 
    $t=\Omega_k(\log n)$, a joint of size $js_{k+1}(G)=\Omega_k (n^{k-1})$ and 
    a generalized book $B_{k,t}$ with size $t=\Omega_k(n)$. 
    Our result can be interpreted as a complement to the result of Nikiforov and the result of Li--Liu--Zhang.  
    
    \item 
    We prove the corresponding stability result for the signless Laplacian spectral radius: Let $F$ be a graph with chromatic number $\chi (F)=k+1\ge 4$. For every $\varepsilon >0$, 
    there exist $n_0$ and $\sigma >0$ such that if $G$ is an $F$-free graph on $n\ge n_0$ vertices with $q(G)\ge (1- \frac{1}{k} - \sigma)2n$, then $G$ can be obtained from $T_{n,k}$ by adding and deleting at most $\varepsilon n^2$ edges.  
\end{itemize} 
Using the probabilistic method, 
we show that the above bounds are optimal up to a constant factor. As applications, we first prove that if $k\ge 2$ and $G$ is an $n$-vertex graph with $m$ edges such that $\sum_{v\in V(G)}d^2(v) \ge 2(1- \frac{1}{k}+ \varepsilon )mn$, then $G$ contains a blowup $K_{k+1}[t]$ with $t=\Omega_{k,\varepsilon}(\log n)$. This generalizes the aforementioned Bollob\'{a}s--Erd\H{o}s theorem. 
{Secondly, we establish the spectral extremal results for hypergraphs by showing that for any graph $F$ with $\chi (F)=k+1$, where $k\ge r\ge 2$, and for any $\varepsilon >0$, there exists $n_0$ such that if $H$ is a linear $r$-uniform hypergraph on $n\ge n_0$ vertices with $q(H)\ge (1- \frac{1}{k} + \varepsilon )\frac{2n}{r-1}$, then $H$ contains an $r$-expansion of $F$.} 
\end{abstract}

 

\newpage 

\section{Background}

Throughout the paper, we will use the following notation. 
Let $f(n)$ and $g(n)$ be two functions of $n$. 
We write $f(n)=\Omega (g(n))$ if 
there exist positive constants $c$ and $n_0$ such that $ f(n)\ge c g(n) >0$ for all $n \ge n_0$. 
Moreover, if the constant $c$ depends on parameters $k$ and $\varepsilon$, then we denote 
$f(n)=\Omega_{k,\varepsilon}(g(n))$. 
Similarly, we write $f(n)=O(g(n))$ if there exist positive constants $C$ and $n_0$ such that $0< f(n)\le Cg(n)$ for all $n\ge n_0$. 
In addition, we write $f(n)=o(g(n))$ if $ f(n)/g(n) \to 0$ as $n\to \infty$, which implies that $f(n)$ is much smaller than $g(n)$ for sufficiently large $n$. 
 The main conclusions of this article aim to study the asymptotic behavior of some graph parameters. 
When we use the above notation, we always admit the assumption that $n$ is a sufficiently large integer. 

 For a graph $G$, we always admit 
$|V|=n$ and $|E|=m$ if there is no confusion. 
Let $K_{k+1}$ be a complete graph on $k+1$ vertices and $K_{t,t}$ be a complete bipartite graph with each partite set of size $t$. 
For a graph $H$, we denote by $H[t]$ the $t$-blowup of $H$, which is obtained by replacing every vertex of $H$ by an independent set of size $t$ and every edge of $H$ by a copy of $K_{t,t}$. 
For example, the complete $k$-partite graph $K_{t,t,\ldots ,t}$ can be viewed as a $t$-blowup of $K_k$. 
 A graph $G$ is called {\it $F$-free} if it  does not contain a subgraph isomorphic to $F$. 
 The {\it Tur\'an number} $\mathrm{ex}(n, {F})$ is defined to be the maximum number of edges of {an $n$-vertex $F$-free graph.}  
We denote by $\mathrm{Ex}(n,{F})$ the set of all $n$-vertex $F$-free graphs with maximum number of edges.

\subsection{Classical extremal graph results}

 In 1941, Tur\'{a}n \cite{Turan41} posed the  natural question of determining 
 $\mathrm{ex}(n,K_{k+1})$ for every integer $k\ge 2$. 
The Tur\'{a}n graph $T_{n,k}$ is defined to be the complete $k$-partite graph on $n$ vertices where 
 its part sizes are as equal as possible. 
 Particularly, we have $T_{n,2}=K_{\lfloor n/2\rfloor, \lceil n/2\rceil}$. 
Tur\'{a}n \cite{Turan41} 
(also see \cite[p. 294]{Bollobas78}) obtained that if $G$ is an $n$-vertex $K_{k+1}$-free graph, 
then $e(G)\le e(T_{n,k})$, with equality if and only if $G=T_{n,k}$. Consequently, we get 
\begin{equation}
    \label{eq-Turan}
    \mathrm{ex}(n,K_{k+1})= e(T_{n,k})=  
\left(1- \frac{1}{k} \right) 
\frac{n^2}{2} - O(1).
\end{equation}   
There are many extensions and generalizations on Tur\'{a}n's theorem.  
The problem of determining $\mathrm{ex}(n, F)$ is usually called the 
Tur\'{a}n-type extremal problem. 
 The most celebrated extension always attributes to a well-known result of Erd\H{o}s, Stone and Simonovits \cite{ES1946,ES1966}. 
 The {\it chromatic number} of a graph $G$, denoted by $\chi(G)$, is the minimum number of colors required to color the vertices of a graph $G$ in such a way that no two adjacent vertices share the same color. 

 \begin{theorem}[See \cite{ES1946,ES1966}] 
 \label{eq-ESS}
      If $F$ is a graph with chromatic number $\chi (F)=k+1 \ge 2$, then 
 \begin{equation*} 
  \mathrm{ex}(n,F) = \left( 1- \frac{1}{k}\right) 
  \frac{n^2}{2} + o(n^2).  
  \end{equation*}
 \end{theorem}
 This provides an asymptotic estimate for the Tur\'{a}n numbers of non-bipartite graphs. 
 The Erd\H{o}s--Stone--Simonovits theorem is commonly regarded as the fundamental result of extremal graph theory; see \cite{BE1973,BES1976,BK1994,CS1981,Ish2002} and references therein.  
 However, for a bipartite graph $F$, where $k=1$, Theorem \ref{eq-ESS} only gives  $\mathrm{ex}(n,F)=o(n^2)$. 
There have been numerous attempts to find better bounds on $\mathrm{ex}(n,F)$ for various bipartite graphs $F$. We refer to 
  the comprehensive surveys \cite{FS13,Sim13}.

\subsection{Adjacency spectral radius}

The {\it adjacency matrix} of a graph $G$ is defined as $A(G)=(a_{ij})_{n\times n}$, where $a_{ij}=1$ if $ij\in E(G)$, and $a_{ij}=0$ otherwise. The {\it adjacency spectral radius} of $G$, denoted by $\lambda (G)$, is the maximum modulus of the eigenvalues of $A(G)$. 
The spectral Tur\'{a}n-type problem seeks to determine the maximum or minimum eigenvalues of graphs that avoid certain substructures. 
It is a central topic in extremal graph theory, combining spectral graph theory and Tur\'{a}n-type questions. This area bridges combinatorial extremal problems and spectral analysis, with applications in network science and discrete geometry. Correspondingly, we define 
$\mathrm{ex}_{\lambda}(n,F)$ to be the maximum spectral radius of an $F$-free graph of order $n$. This is also called the spectral Tur\'{a}n number.

Extending the classical Tur\'{a}n theorem, Wilf \cite{Wil1986} proved that if $G$ is a $K_{k+1}$-free graph on $n$ vertices, then $ \lambda (G) \le \left( 1- \frac{1}{k}\right)n$.  
Moreover, Nikiforov \cite{Niki2007laa2} and Guiduli \cite{Gui1996} independently proved further that 
$\lambda (G)\le \lambda (T_{n,k})$. 
From their results, we have  
\begin{equation} \label{eq-a-Turan}
 \mathrm{ex}_{\lambda}(n,K_{k+1}) = \lambda (T_{n,k})=
 \left( 1-\frac{1}{k} \right)n - 
 O(\frac{1}{n}). 
\end{equation}
This implies the Tur\'{a}n bound in (\ref{eq-Turan}) immediately by the fact $\lambda (G)\ge \frac{2e(G)}{n}$. More generally, 
Nikiforov \cite{ESB2009} and
Guiduli \cite{Gui1996} independently proved a spectral extension of the Erd\H{o}s--Stone--Simonovits theorem in terms of the adjacency spectral radius. 

\begin{theorem}[See \cite{Gui1996,ESB2009}]
 \label{eq-a-ESS}  
 If $F$ is a graph with chromatic number $\chi (F)=k+1 \ge 2$, then  
\begin{equation*}
    \mathrm{ex}_{\lambda}(n,F) 
    = \left(1- \frac{1}{k}\right)n + o(n). 
\end{equation*} 
\end{theorem}

 Boosted by the study of Tur\'{a}n-type problem, 
 the spectral Tur\'{a}n-type problem has rapidly emerged as a focal point of research in extremal graph theory and has yielded a lot of significant results.  
 Similar spectral extremal problems explore forbidden paths, cycles, or other configurations, linking graph eigenvalues to structural constraints. We refer the readers
 to the surveys \cite{LiLF,Niki2011}.

\subsection{Signless Laplacian spectral radius}

The {\it signless Laplacian matrix} of a graph $G$ 
is defined as $Q(G)=D(G)+A(G)$, 
where $D(G)$ and $A(G)$ are the degree diagonal matrix and the adjacency matrix of $G$, respectively. 
The largest eigenvalue of $Q(G)$ is denoted by  $q(G)$, which is called the {\it signless Laplacian spectral radius} or {\it $Q$-index} of $G$.  
We define $\mathrm{ex}_{q}(n,F)$ to be the maximum signless Laplacian spectral radius of an $n$-vertex $F$-free graph.  In 2013, de Abreu and Nikiforov \cite{AN2013} proved that if $G$ is an $n$-vertex $K_{k+1}$-free graph, then $ q(G) \le \left( 1-\frac{1}{k}\right)2n$. 
Furthermore, He, Jin, and Zhang \cite{HJZ2013} showed that such a graph $G$ satisfies $q(G)\le q(T_{n,k})$. Consequently, it follows that  
\begin{equation} \label{eq-q-Turan}
    \mathrm{ex}_q(n,K_{k+1}) = q(T_{n,k})= \left( 1- \frac{1}{k}\right)2n - O(\frac{1}{n}). 
\end{equation}
Recently, Zheng, Li and Su \cite{ZLS2025} showed the following result for a general graph. 

\begin{theorem}[See \cite{ZLS2025}] \label{eq-q-ESS} 
    If $F$ is a graph with $\chi (F)=k+1 \geq 3$, then
\begin{equation*} 
    \mathrm{ex}_{q}(n,F)=\bigg(1-\frac{1}{k}\bigg)2n + o(n). 
\end{equation*} 
\end{theorem} 
 Theorem \ref{eq-q-ESS} extends both Theorems \ref{eq-ESS} and \ref{eq-a-ESS} by invoking the fact $q(G)\ge 2\lambda (G) \ge \frac{4e(G)}{n}$. 
This demonstrates that the signless Laplacian spectral radius provides a unified framework that incorporates both the size and adjacency spectral radius. 
In this paper,  
we focus primarily on deriving some new Tur\'{a}n-type results involving the signless Laplacian spectral radius. 

Before our work, there are many spectral extremal graph  results that involve the signless Laplacian spectral radius for different types of subgraphs $F$. 
To complete a comprehensive review, 
we now summarize as much progress as possible in this area. 
For example, we refer the interested readers to the results for matchings \cite{Yu2008}, 
paths \cite{NY1}, odd cycles \cite{Y2014}, even cycles \cite{NY2}, 
quadrilaterals \cite{FNP2013},  
complete bipartite graphs \cite{FNP2016}, 
Hamilton cycles \cite{LN2016,LLP2018}, 
Hamilton-connected graphs \cite{ZWL2020}, 
long cycles \cite{CWZ2022}, 
non-bipartite graphs forbidden short odd cycles \cite{LMX2022,MLX2022}, 
linear forests \cite{CLZ2020}, trees \cite{CLZ2022}, 
friendship graphs \cite{ZHG2021}, 
flowers \cite{CZ}, fan \cite{WZ2023}, 
sparse spanning graphs \cite{LN2023}, 
books \cite{CJZ2025}, color-critical graphs \cite{ZLL2025} and extremal problems for graphs with given number of edges \cite{LGW2021,ZXL2020,ZXL2022}.

\section{Main results}  
\label{sec-2}

\subsection{Supersaturation and large blowup of cliques}

For two positive functions $f(n)$ and $g(n)$, we denote $f(n)= \Omega_{k,\varepsilon}(g(n))$ if there exist constants $c$ and $n_0$ depending on $k,\varepsilon$ such that $f(n)\ge cg(n)$ for all $n\ge n_0$. 

The Erd\H{o}s--Simonovits  supersaturation \cite{ES1983-super} states that for any $k\ge 1$ and $\varepsilon>0$, if $G$ is an $n$-vertex graph with $e(G)\ge (1- \frac{1}{k} + \varepsilon) \frac{n^2}{2}$, then $G$ 
contains $\Omega_{k,\varepsilon} (n^{k+1})$ copies of $K_{k+1}$. 
This result can be proved by different approaches. The first approach uses the standard double counting argument, which shows that for each $t\ge k+1$, the graph $G$ contains at least $\frac{\varepsilon}{3} \binom{n}{t}$ sets $T$ of size $t$ in $V(G)$ such that 
$e(G[T]) \ge (1-\frac{1}{k} + \frac{\varepsilon}{2} ) 
\frac{t^2}{2}$. Each of these $t$-sets contains a copy of $K_{k+1}$. So $G$ contains at least $\frac{\varepsilon}{3} 
\binom{n}{t} \big/ 
\binom{n- k-1}{t-k-1 }$ copies of $K_{k+1}$; see \cite[Lemma 2.1]{Kee2011}. 
The second approach uses the vertex-deletion argument, which shows that $G$ contains a subgraph $G'$ on $n'\ge  \frac{1}{4}\sqrt{\varepsilon}n$ vertices with  minimum degree $\delta (G')\ge (1-\frac{1}{k}+\frac{\varepsilon }{2})n'$; see \cite[Proposition 4.2]{Kee2011}. 
By the greedy algorithm, we can find $\Omega_{k,\varepsilon}((n')^{k+1})=\Omega (n^{k+1})$ copies of $K_{k+1}$ in the subgraph $G'$.  
The third approach is to use the Moon--Moser theorem, which counts the number of cliques iteratively and implies that $G$ has at least $\frac{\varepsilon}{(k+1)^k} n^{k+1}$ copies of $K_{k+1}$; see, e.g., \cite{Niki2008blms}.

Let $\# K_{k+1}$ be the number of copies of $K_{k+1}$ in a graph $G$. 
In 2007, Bollob\'{a}s and Nikiforov \cite{BN2007jctb} established the spectral supersaturation for cliques by proving that for every $k\ge 2$, 
\begin{equation} \label{eq-BN-spectral-counting}
\# K_{k+1} \ge \left( \frac{\lambda (G)}{n}-1+\frac{1}{k}\right)\frac{k(k-1)}{k+1}
\left( \frac{n}{k}\right)^{k+1}. 
\end{equation}
As an application of (\ref{eq-BN-spectral-counting}), one can easily obtain the following result.

\begin{theorem}[Bollob\'{a}s--Nikiforov \cite{BN2007jctb}] \label{thm-BN}
    If $k\ge 1,\varepsilon >0$ and $G$ is an $n$-vertex graph with 
    \[ \lambda (G) 
    \ge \left( 1- \frac{1}{k} + 
    \varepsilon \right)n, \]   
then $G$ contains $\Omega_{k,\varepsilon}  (n^{k+1})$ copies of $K_{k+1}$. 
\end{theorem}

The spectral supersaturation for counting general graphs (rather than cliques) could be obtained by using the graph removal lemma. We refer to a paper of Li, Lu and Peng \cite[Sec. 5]{LLP2024-AAM}.

As the first new result of this paper, we give an extension of Theorem \ref{thm-BN} and prove the following supersaturation in terms of the signless Laplacian spectral radius. 
Our proof is based on the vertex-deletion argument and develops a different method from that of Theorem \ref{thm-BN}.

\begin{theorem} \label{thm-1-2}
    If $k\ge 2,\varepsilon>0$ and $G$ is an $n$-vertex graph with 
    \[ q(G)\ge \left(1-\frac{1}{k} + \varepsilon \right)2n, \]
    then $G$ contains $\Omega_{k,\varepsilon} (n^{k+1})$ copies of $K_{k+1}$. 
\end{theorem}

\noindent 
{\bf Remark.} 
 Theorem \ref{thm-1-2}
extends Theorem \ref{thm-BN} by the fact $q(G)\ge 2\lambda (G)$. Moreover, 
the bounds in both Theorem \ref{thm-BN} and \ref{thm-1-2} are optimal up to a constant factor by embedding $\varepsilon n^2$ edges to a partite set of the $k$-partite Tur\'{a}n graph $T_{n,k}$ such that there is no copy of $K_{k+1}$ within this partite set. This is possible since the partite set contains $n/k$ vertices, where we may assume that $k$ divides $n$. Then the resulting graph $G$ has exactly $(1- \frac{1}{k})\frac{n^2}{2} + \varepsilon n^2$ edges. Therefore, we have $\lambda (G)\ge \frac{2m}{n} > (1- \frac{1}{k})n + \varepsilon n$ and $q(G)\ge (1- \frac{1}{k})2n + 2\varepsilon n$. While $G$ contains $\varepsilon n^2 \cdot (\frac{n}{k})^{k-1} = O_{k,\varepsilon} (n^{k+1})$ copies of $K_{k+1}$.

Recall that $K_{k}[t]$ denotes 
{\it the $t$-blowup} of $K_k$, which is 
the complete $k$-partite graph where each partite set has size $t$.  
In 1946, Erd\H{o}s and Stone \cite{ES1946} proved that for fixed $k\ge 2, t\ge 1,  \varepsilon >0$ and sufficiently large $n$, if $G$ is an $n$-vertex graph with $e(G)\ge (1- \frac{1}{k} + \varepsilon) \frac{n^2}{2}$, then $G$ contains a copy of $K_{k+1}[t]$. Furthermore, 
Bollob\'as and Erd\H{o}s \cite{BE1973} proved that the size $t=\Omega_{k,\varepsilon}(\log n)$. 
There are many results that focus on improving the bound on $t$. Bollob\'{a}s, Erd\H{o}s and Simonovits \cite{BES1976} 
proved that $t\ge \frac{c \log n}{k \log (1/\varepsilon)}$ for some absolute constant $c >0$.  
We refer to \cite{CS1981,BK1994,Ish2002} for related results.

In 2009, Nikiforov \cite{ESB2009} proved the following spectral generalization.

\begin{theorem}[Nikiforov \cite{ESB2009}]\label{sES}
If $k\geq 1$, $\varepsilon >0$ and $G$ is an $n$-vertex graph with 
\[ \lambda(G)\geq \left(1-\frac{1}{k}+\varepsilon \right)n, \]  
then $G$ contains a copy of $K_{k+1}[t]$ with $t=\Omega_{k,\varepsilon} (\log n)$.
\end{theorem}

As the second new result, we prove the following extension of Theorem \ref{sES}.

\begin{theorem}\label{lsES}
If $k\geq 2$, $\varepsilon >0$ and $G$ is an $n$-vertex graph with 
\[ q(G)\geq \left(1-\frac{1}{k} 
+\varepsilon \right)2n, \]  
then $G$ contains a copy of $K_{k+1}[t]$ with $t=\Omega_{k,\varepsilon} (\log n)$.
\end{theorem}

\noindent 
{\bf Remark.}
We remark that 
the logarithmic dependence on $n$ in Theorems \ref{sES} and \ref{lsES} is optimal by considering random graphs of the same edge density; see Example \ref{exam-2-4} for details.

Observe that for any graph $F$ with $\chi (F)=k+1$, we can find an integer $t$ such that 
$F$ is a subgraph of $K_{k+1}[t]$. 
Theorem \ref{lsES} implies if $q(G)\ge (1- \frac{1}{k} + \varepsilon)2n$, then for sufficiently large $n$, we can find a copy of $F$ in $G$. Thus, Theorem \ref{lsES} is a generalization of Theorem \ref{eq-q-ESS}. 
Building on this viewpoint, we provide an alternative proof for Theorem \ref{eq-q-ESS}, different from the technique in \cite{ZLS2025}. 

We point out that Theorems \ref{thm-BN} and \ref{sES} are valid for every $k\ge 1$ although the original references have written the assumption $k\ge 2$. 
In fact, in the case $k=1$, if $G$ is an $n$-vertex graph with $\lambda (G)\ge \varepsilon n$, then combining with $\lambda (G) < \sqrt{2m}$, we have $m > \frac{1}{2}\varepsilon^2n^2$, that is, $G$ has positive edge density. Furthermore, the K\H{o}vari--S\'{o}s--Tur\'{a}n theorem (see, e.g., \cite[p. 311]{Bollobas78}) implies that $G$ contains a copy of $K_2[t]$ with $t=\Omega_{\varepsilon}(\log n)$. 
So Theorems \ref{thm-BN} and \ref{sES} hold in the case $k=1$.

We emphasize that both Theorems \ref{thm-1-2} and \ref{lsES} do {\bf not} hold in the case $k=1$, since $q(G)\ge 2\varepsilon n$ can not guarantee that $G$ has positive edge density.  
For example, taking $s\in \mathbb{N}$ as a fixed integer, we know that $q(K_{s,n-s}) = n \ge 2 \varepsilon n$ for every $\varepsilon \le \frac{1}{2}$, but $K_{s,n-s}$ has exactly $s(n-s)=O(n)$ edges and it does not contain a copy of $K_{2}[t]$ with $t=\Omega_{\varepsilon}(\log n)$, since $s$ is a fixed integer.

\subsection{Finding large color-critical graphs}

A graph $F$ is called {\it color-critical} if   there exists an edge $e$ of $F$ such that $\chi(F-e)<\chi(F)$, where $F-e$ denotes the graph obtained from $F$ by removing the edge $e$. 
This is a very broad and important class of graphs. 
For example, cliques, odd cycles, wheels with even order, cliques with one edge removed, complete bipartite graphs plus an edge, books and joints are color-critical; see \cite{Mub2010,PY2017,ZL2022jgt, LLZ2025+}. 
Extending the Tur\'{a}n theorem in (\ref{eq-Turan}), 
Simonovits \cite{S1968} proved that 
if $k\ge 2$ and $F$ is a color-critical graph with $\chi(F)=k+1$, then for sufficiently large $n$, every $n$-vertex $F$-free graph has at most 
$e(T_{n,k})$ edges, and $T_{n,k}$ is the unique  graph that achieves this bound.

\subsubsection{Complete multipartite graphs plus an edge}

 Let $K_{k}^{+}[t]$ be the graph obtained from the complete $k$-partite graph $K_{k}[t]$ by adding exactly one edge within a partite set.  
 In 2009, Nikiforov \cite{N2009} proved a spectral extension of the Simonovits result. 

 \begin{theorem}[Nikiforov \cite{N2009}] 
 \label{thm-2-5}
If $k\geq 2$ and $G$ is an $n$-vertex graph with 
\[ \lambda (G)> \lambda ( T_{n,k}), \]
then $G$ contains a copy of 
$K_{k}^{+}[t]$ with $t=\Omega_k( \log n)$.
\end{theorem}

Next, we present an analogue for the signless Laplacian spectral radius. 

\begin{theorem}\label{thm-h3}
If $k\geq 3$ and $G$ is an $n$-vertex graph with 
\[ q(G)\ge q( T_{n,k}), \] 
then $G$ contains a copy of 
$K_{k}^{+}[t]$ with $t=\Omega_k(\log n)$, unless $G=T_{n,k}$.
\end{theorem}

\noindent 
{\bf Remark.} 
The bounds in Theorems \ref{thm-2-5} and \ref{thm-h3} are optimal; see Example \ref{exam-2-6}. 
 Note that every color-critical $F$ with $\chi (F)=k+1$ is a subgraph of $K_{k}^{+}[t]$ for some integer $t$.  
Thus, Theorem \ref{thm-h3} implies that for 
any color-critical graph $F$ with $\chi(F)=k+1\ge 4$, if $n$ is sufficiently large, then 
$\mathrm{ex}_q(n,F) = q (T_{n,k})$, and $T_{n,k}$ is the unique extremal graph. This recovers a recent result of Zheng, Li and Li \cite{ZLL2025}.

\subsubsection{Joints}

{\it The book of size $t$}, denoted by $B_t$, is a graph that consists of $t$ triangles sharing a common edge. 
The booksize $bk(G)$ of a graph $G$ is the size of a largest book contained in $G$. 
Erd\H{o}s \cite{Erd1962a} proved that every graph on $n$ vertices with $\lfloor n^2/4\rfloor +1$ edges satisfies $bk(G) = \Theta(n)$ and conjectured that $bk(G)> n / 6$, which was confirmed by Edwards (see \cite[Lemma 4]{EFR1992}) 
and independently 
by Khad\v{z}iivanov and Nikiforov  \cite{KN1979}.   
Unfortunately, neither of the original references can be found.  
Two different proofs were provided by 
Bollob\'{a}s and Nikiforov \cite{BN2005} and Li, Feng and Peng \cite[Sec. 4.3]{LFP2024-triangular}.

The spectral extremal problem for books was widely investigated in recent years. 
In 2023, Zhai and Lin \cite{ZL2022jgt} proved that if $G$ is an $n$-vertex graph with $\lambda (G)> \lambda (T_{n,2})$, then $bk(G)> 2n/13$. Interestingly, they \cite[Problem 1.2]{ZL2022jgt} speculated that this can be improved to $bk(G)> n/6$. This is still unsolved. 
Confirming the Zhai--Lin--Shu conjecture  \cite{ZLS2021}, Nikiforov \cite{Niki2021} proved that if $G$ is an $m$-edge graph with $\lambda (G) > \sqrt{m}$, then $bk(G) =\Omega(m^{1/4})$. 
Furthermore, Li, Liu and Zhang \cite{LLZ2024-book-4-cycle} improved this bound by showing that  $bk(G)=\Omega({m}^{1/2})$ and the exponent $1/2$ is the best possible. 

In what follows, 
we turn attention to the study of two fundamental extensions of books, namely, joints and generalized books.  
An \emph{$r$-joint of size $t$} is a collection of $t$ cliques of order $r$ sharing a common edge. Let $js_r(G)$ denote the maximum size of an $r$-joint of $G$. For example, we have $js_3(G) = bk(G)$. 
An old result of Erd\H{o}s \cite{Erd1969} 
shows that every $n$-vertex graph $G$ with $e(G) > (1 - \frac{1}{k}) \frac{n^2}{2}$ contains not only one copy of  $K_{k+1}$, but also $\Omega_k(n^{k-1})$ copies of $K_{k+1}$ that share a common edge. In other words, we have $js_{k+1}(G)=\Omega_k(n^{k-1})$; 
see \cite{BN2008,BN2011} for related results. 
In 2009, Nikiforov \cite{N2009} provided a spectral extension by showing that if $\lambda (G)> (1- \frac{1}{k})n$, then $js_{k+1}(G)=\Omega_k(n^{k-1})$. 
Extending the result of Nikiforov \cite{Niki2021}, Li, Liu and Zhang \cite{LLZ2024-book-4-cycle} proved the following spectral Tur\'{a}n extension for graphs with given number of edges, instead of the number of vertices.

\begin{theorem}[Li--Liu--Zhang \cite{LLZ2024-book-4-cycle}] \label{thm-jsr+1}
    \label{thm:joints} 
    If $k\ge 2$ and $G$ is an $m$-edge graph with  
    \[  \lambda^2(G) > 
    \left(1 - \frac{1}{k} \right) 2m, \]  
    then $js_{k + 1}(G) = 
    \Omega_k (m^{\frac{k-1}{2}})$. 
    The bound is optimal up to a constant factor. 
\end{theorem}

Inspired by the above result, 
we establish the following analogue. 

\begin{theorem} \label{thm-q-js}
    If $k\ge 3$ and $G$ is an $n$-vertex graph with 
    \[ q(G) \ge  q(T_{n,k}), \]  
    then $js_{k+1}(G)=\Omega_k (n^{k-1})$,  
   unless $G=T_{n,k}$. 
\end{theorem}

\noindent 
{\bf Remark.} 
 This bound is optimal up to a constant factor, which can be seen by adding an edge to a partite set of the $k$-partite Tur\'{a}n graph $T_{n,k}$. Then $js_{k+1}(G)=(\frac{n}{k})^{k-1}$.

\subsubsection{Generalized books}

Another way to extend the results on books is to 
consider {\it the generalized book} $B_{k,t}$, which 
is a graph obtained from $t$ copies of 
$K_{k+1}$ sharing a common $K_k$. 
In other words, we have $B_{k,t}:=K_k\vee I_t$, which is obtained by joining every vertex of a clique $K_k$ to every vertex of an independent set $I_t$ of size $t$. Particularly, we have $B_{2,t}=B_t$. Recently, 
Li, Liu and Zhang \cite{LLZ2024-book-4-cycle} proved the following result, which gives a lower bound on the largest generalized booksize.

\begin{theorem}[Li--Liu--Zhang \cite{LLZ2024-book-4-cycle}] \label{thm-LLZ2024}
If $k\ge 2$ and $G$ is an $m$-edge graph with 
\[ \lambda^2(G) > \left(1 - \frac{1}{k} \right)  2m, \]  
then $G$ has a copy of $B_{k, t}$ with $t = \Omega_k({m}^{1/2})$. The bound is optimal up to a constant factor. 
    \end{theorem}
    
As pointed out in \cite{LLZ2024-book-4-cycle}, 
Theorem \ref{thm-LLZ2024} implies that if $G$ is an $n$-vertex graph with the following condition: 
either $m> (1-\frac{1}{k})\frac{n^2}{2}$ edges or 
$\lambda (G)> (1-\frac{1}{k})n$, 
then $G$ contains a copy of $B_{k,t}$ 
with $t=\Omega_k(n)$. This can be regarded as an extension of a result of Zhai and Lin \cite{ZL2022jgt}, in which they obtained that if $\lambda (G) > \lambda (T_{n,2})$, then $G$ contains a book $B_{2,t}$ with $t=\Omega(n)$.

Motivated by Theorem \ref{thm-LLZ2024}, we present the following variant.

\begin{theorem} \label{thm-genelized-books}
If $k\ge 3$ and $G$ is an $n$-vertex graph with 
\[ q(G) \ge q(T_{n,k}), \] 
then $G$ has a copy of $B_{k,t}$ with $t= \Omega_k (n)$, unless $G=T_{n,k}$. 
\end{theorem}

\noindent 
{\bf Remark.} 
 The optimality of the bound follows from 
 the construction by adding an edge to a partite set of $T_{n,k}$. So the resulting graph contains a copy of $B_{k,t}$ with $t=\frac{n}{k}$.

The assumption $k\ge 3$ is {\bf required} in Theorems \ref{thm-h3}, \ref{thm-q-js}, and \ref{thm-genelized-books}.
In the case $k=2$, let $s\ge 1$ be fixed and $K_{n-s,s}^+$ be a graph obtained from $K_{n-s,s}$ by embedding an edge to the partite set of size $n-s$. We see that $q(K_{n-s,s}^+) > q(T_{n,2})$, but $K_{n-s,s}^+$ has no copy of $B_{2,t}$ with $t=\Omega (n)$.

\subsection{Stability results via spectral radius}

In 1966, 
Erd\H{o}s \cite{Erd1966Sta1,Erd1966Sta2} and Simonovits \cite{S1968} independently  proved a stronger structural theorem by showing that the Tur\'{a}n problem exhibits a certain stability phenomenon: 
Let $F$ be a graph with $\chi (F)=k+1\ge 3$. 
For every $\varepsilon >0$, 
there exist $\delta >0$ and $n_0$ such that 
if $G$ is an $F$-free graph on $n\ge n_0$ vertices with $e(G)\ge (1- \frac{1}{k} - \delta) \frac{n^2}{2}$, then $G$ can be obtained from $T_{n,k}$ 
by adding and deleting at most $\varepsilon n^2$ edges. Roughly speaking, if $G$ is an $n$-vertex 
 $F$-free graph for which $e(G)$ is close to $e(T_{n,k})$, then the structure of $G$ must resemble the Tur\'{a}n graph $T_{n,k}$ in an appropriate sense. 
The stability theorem has proven to be an effective method for determining the exact Tur\'{a}n numbers of non-bipartite graphs in recent years; see \cite[Sec. 5]{Kee2011} and references therein.

 In 2009, Nikiforov \cite{Niki2009jgt} extended the aforementioned stability result of Erd\H{o}s and Simonovits in terms of the adjacency spectral radius of a graph and proved the following result.

\begin{theorem}[Nikiforov \cite{Niki2009jgt}] \label{thm-Nik-stability}
Let $k\ge 2$ and $F$ be a graph with $\chi (F)=k+1$. For every $\varepsilon >0$,  
there exist $\delta >0$ and $n_0$ such that 
if $G$ is an $F$-free graph on $n\ge n_0$ vertices with 
$\lambda(G)\ge (1- \frac{1}{k} - \delta) n$, 
 then $G$ can be obtained from $T_{n,k}$ 
by adding and deleting at most $\varepsilon n^2$ edges.  
\end{theorem}

The spectral stability theorem has emerged as a pivotal and fundamental tool in spectral extremal graph theory. Its significance is underscored by its critical role in resolving numerous spectral extremal problems, as evidenced by recent applications in works such as \cite{DKLNTW2021, FTZ2024, NWK2022, Wang2022}. 

We give an extension of Theorem \ref{thm-Nik-stability} in terms of the signless Laplacian spectral radius.

\begin{theorem}\label{sst}
Let $k\ge 3$ and $F$ be a graph with $\chi(F)=k+1$. For every $\varepsilon>0$, there exist $\sigma>0$ and $n_{0}$ such that
if $G$ is an $F$-free graph on $n\geq n_{0}$ vertices with $q(G)\geq2(1- \frac{1}{k}-\sigma)n$,  then $G$ can be obtained from $T_{n,k}$ 
by adding and deleting at most $\varepsilon n^2$ edges. 
\end{theorem}

\noindent 
{\bf Remark.}
Theorem \ref{sst} does {\bf not} hold for tri-partite graphs $F$. Indeed, we take $F=C_{2k+1}$ 
or $F_k$ for every $k\ge 2$, where $C_{2k+1}$ is an odd cycle of order $2k+1$, and $F_k$ is the friendship graph that consists of $k$ triangles intersecting in a common vertex. 
Note that $\chi (C_{2k+1})=\chi (F_k)=3$. 
Let $S_{n,k}$ be the graph consisting of a clique on $k$ vertices and an independent set on $n-k$ vertices in which each vertex of the clique is joined to each vertex of the independent set. 
We see that $S_{n,k}$ is $C_{2k+1}$-free and $F_k$-free, and $q(S_{n,k}) \ge (\frac{1}{2} - o(1))2n$. However, $S_{n,k}$ can not be obtained from 
$T_{n,2}$ by adding and deleting $o(n^2)$ edges, since  $e(S_{n,k})={k \choose 2} + k(n-k)$ and $e(T_{n,2})= \lfloor n^2/4 \rfloor$.

\subsection{Application I: Degree powers in graphs}

 Given a real number $p\ge 1$, the degree power of $G$ is defined as $\sum_{v\in V(G)} d^p(v)$. 
 In the case $p=1$,
we have $\sum_{v\in V(G)} d(v) =2e(G)$.
In the case $p=2$, the value $\sum_{v\in V(G)} d^2(v)$ is referred to as the {\it first Zagreb index} or {\it $\ell_2$-norm}; see, e.g., \cite{LL2009-Zagreb}. 
Motivated by the study of Tur\'{a}n problem, there has been a wide investigation on estimating the degree power of an $F$-free graph $G$. 
For instance, Gerbner \cite{G2025} found an interesting connection between degree powers and counting stars in an $F$-free graph, which yields a new proof of a theorem of Nikiforov \cite{Niki2009-degree}.   
We refer to  \cite{BN2004,BN2012,CY2000,LLQ2019,Niki2009-degree,Zhang2022}.

On the study of Ramsey theory,
Nikiforov and Rousseau \cite{NR2004} extended the  classical Tur\'{a}n theorem by proving that if $G$ is a $K_{k+1}$-free graph on
$n$ vertices with $m$ edges, then
\begin{equation}    \label{thm-NR-degree}
\sum_{v\in V(G)} d^2(v) \le 2\left( 1-\frac{1}{k}\right)mn.
\end{equation}
The original proof of (\ref{thm-NR-degree}) uses the combinatorial technique.
Recently,  Li, Liu and Zhang \cite[Sec. 5.4]{LLZ2024-book-4-cycle} provided two spectral proofs of (\ref{thm-NR-degree}).
Using a similar line of their proofs, 
Zheng, Li and Li \cite{ZLL2025} extended (\ref{thm-NR-degree}) to $F$-free graphs for any color-critical graph $F$. In this section, 
we can give an asymptotic bound on the sum of squares of degrees in an $F$-free graph for a general graph $F$.

\begin{theorem} \label{thm-2-14}
Let $k\ge 2$ and $F$ be a graph with chromatic number $\chi(F)= k+1$.
If $G$ is an $F$-free graph on
$n$ vertices with $m$ edges, then
\[ \sum_{v\in V(G)} d^2(v) \le 2\left( 1-\frac{1}{k} + o(1)\right)mn. \]  
\end{theorem}

\noindent
{\bf Remark.} 
The above bound is tight when we take $G$ as any  complete bipartite graph for $k=2$, or $G$ is the 
balanced complete $k$-partite graph $T_{n,k}$ for every $k\ge 3$. 
 Theorem \ref{thm-2-14} implies the Erd\H{o}s--Stone--Simonovits bound in Theorem \ref{eq-ESS}. 
In fact, using the Cauchy--Schwarz inequality, we have
$\sum_{v\in V} d^2(v) \ge \frac{1}{n} 
\left(\sum_{v\in V} d(v) \right)^2 = \frac{4m^2}{n}$, which together with 
Theorem \ref{thm-2-14} yields $\frac{4m^2}{n} \le 2\left(1- \frac{1}{k} + o(1) \right)mn$. Then we get  
$m\le \left(1- \frac{1}{k} + o(1) \right) \frac{n^2}{2}$, as expected. 

\medskip 
In addition, Theorem \ref{thm-2-14} also implies the following corollary.

\begin{coro} \label{coro-1-7}
Let $k\ge 2$ and $F$ be a graph with chromatic number $\chi(F) =k+1$.
If $G$ is an $F$-free graph on
$n$ vertices, then
\[ \sum_{v\in V(G)} d^2(v) \le \left( 1-\frac{1}{k} + o(1)\right)^2 n^3. \] 
\end{coro}

\noindent
{\bf Remark.} 
The above bound is tight when we take $G$ as the $k$-partite Tur\'{a}n graph $T_{n,k}$. 
This corollary recovers a result of Bollob\'{a}s and Nikiforov \cite{BN2012} by a quite different method. 
 Corollary \ref{coro-1-7} requires $\chi (F)\ge 3$. 
For forbidding bipartite graphs, we refer the readers to 
\cite{BL2024,G2025,LLQ2019,Niki2009-degree}.

\subsection{Application II: Linear hypergraphs}
A {\it hypergraph} $H=(V,E)$ consists of a vertex set $V=\{v_1,v_2,\ldots,v_n\}$  and an edge set $E=\{e_1,e_2$, $\ldots,e_m\}$, where $e_i \subseteq V$ for $i \in [m]:=\{1,2,\ldots,m\}$.
 Let $|H|:=|V(H)|$ and $e(H):= |E(H)|$ denote  the number of vertices and edges in $H$, respectively. If $|e_i|=r$ for each $i \in [m]$ and $r \geq2$, then $H$ is called an  {\it $r$-uniform} hypergraph ({$r$}-graph, for short). A hypergraph is called {\it linear} if any two edges intersect in at most one vertex.   
 The adjacency spectral problems for (linear) hypergraphs were studied in \cite{GC2021,GCH2022,HCC2021,KLM2014}.  
  The  {\it $r$-expansion}  of a $2$-graph $F$ is the $r$-graph $F^{r}$ obtained from $F$ by enlarging each edge of $F$ with $(r-2)$ new vertices disjoint from $V(F)$ such that distinct edges of $F$ are enlarged by distinct vertices; 
 see, e.g., \cite{Ger2025,SFKH2022,SFK2025} for related results.   
 
  In this section, we shall study the signless Laplacian spectral Tur\'an-type problem of linear hypergraphs. 
Given a vector $\mathbf{x}=(x_{1},\ldots,x_{n})\in \mathbb{R}^{n}$, the $r$-norm of $\mathbf{x}$ is denoted by  $\|\mathbf{x}\|_{r}:=(|x_{1}|^{r}+\cdots +|x_{n}|^{r})^{1/r}$. 
Given an $r$-graph $H$ on $n$ vertices and an edge $e\in E(H)$, we denote $x^{e}:=\prod_{v\in e}x_{v}$. The {signless Laplacian spectral radius} of a $r$-graph $H$ is defined as
$$q(H):=\max_{\|\mathbf{x}\|_{r}=1} \left( \sum_{v\in V(H)}d_{H}(v)x_{v}^{r}+  r\sum_{e\in E(H)}x^{e}\right).$$ 
If $\mathbf{x}\in \mathbb{R}^{n}$ is a vector such that $\|\mathbf{x}\|_{r}=1$ and
$q(H)=\sum_{v\in V(H)}d_{H}(v)x_{v}^{r}+  r\sum_{e\in E(H)}x^{e}$, then $\mathbf{x}$ is called an  {eigenvector} of $H$ corresponding to $q(H)$. Clearly, there always exists a nonnegative eigenvector corresponding to $q(H)$. 
 {By establishing  connections between the signless Laplacian spectral radius of a linear hypergraph and its $2$-shadow graph,} 
 we determine the upper bound on $q(H)$ where $H$ does not contain an $r$-expansion of a complete multipartite graph.

\begin{theorem}\label{lhslES}
Let $k\geq 2$, $\ell \ge 1$, $r\geq2$, $0<\varepsilon<1/(k-1)$ and $n$ be sufficiently large. If $H$ is a linear $r$-uniform hypergraph with $n$ vertices satisfying
$$q(H)\geq \frac{2n}{r-1}\bigg(1-\frac{1}{k}+ \varepsilon \bigg),$$
then $H$ contains an $r$-expansion of $K_{k+1}[\ell]$.
\end{theorem}

We define $\mathrm{ex}_{q}^{lin}(n,F)$ to be the maximum signless Laplacian spectral radius of all linear $F$-free $r$-graphs on $n$ vertices. 
By the existence of a specific design (Lemma \ref{design}),  
we present an extension of Theorem \ref{eq-q-ESS} when we forbid an $r$-expansion of a general graph.

\begin{theorem}\label{ESSL}
Let $F$ be a graph with $\chi(F)=k+1$, where $k\geq r\geq2$. Then
$$\mathrm{ex}_{q}^{lin}(n,F^{r})=\frac{2n}{r-1} 
\bigg(1-\frac{1}{k}\bigg)+o(n).$$
\end{theorem}

Moreover, we derive a sharp upper bound on the  signless Laplacian spectral radius of linear hypergraphs that do not contain an expansion of a color-critical graph; see Theorem \ref{linear hypergraph}.

\paragraph{Our approach.}
 We provide a unified approach to proving the 
 new results stated in this section, including the
 supersaturation and blowup of cliques, the existence of large color-critical graphs, joints and generalized books, and the stability theorem for the signless Laplacian spectral radius.  Our framework primarily relies on two powerful  graph reduction algorithms.  These algorithms are based on iteratively deleting the vertices with the smallest coordinates of the Perron eigenvector corresponding to the signless Laplacian spectral radius. Roughly speaking, for a graph $G$ with large signless Laplacian spectral radius, if it contains a `bad' vertex with small degree, then we can compute that $G$ has a small coordinate of the Perron eigenvector. By iteratively deleting the vertex corresponding to the smallest coordinate, 
 we can obtain a subgraph $H$ that has both a large spectral radius and a large minimum degree; see Theorem \ref{h1} for more details.  
 Leveraging some known results concerning the minimum degree conditions, e.g., Lemmas \ref{lem-super-cliques} -- \ref{est}, we can find the desired substructures in the subgraph $H$.

\paragraph{Organization.} 
In Section \ref{sec-3}, we introduce some useful lemmas involving the minimum degree and the minimum entry of eigenvector. 
Two important results (Theorems \ref{h1} and \ref{h1-two-cases}) will be proved in this section. In Section \ref{sec-4}, we will study the supersaturation and blowup of cliques and provide the proofs of Theorems \ref{thm-1-2} and \ref{lsES}. 
In Section \ref{sec-5}, we study the color-critical graphs, joints and generalized books and present the proofs of Theorems \ref{thm-h3}, \ref{thm-q-js} and \ref{thm-genelized-books}. In Section \ref{sec-6}, we show the proof of the stability in Theorem \ref{sst}. In Section \ref{sec-7}, we give two alternative proof of Theorem \ref{thm-2-14}. 
In Section \ref{sec-8}, we study the signless Laplacian spectral radius of hypergraphs. 
{By establishing a key lemma involving a linear hypergraph and its $s$-shadow graph (Lemma \ref{gj})}, we show the proofs of Theorems \ref{lhslES} and \ref{ESSL}. In Section \ref{sec-9}, we conclude with two interesting conjectures for further research.

\paragraph{Notation.} 
Recall that $K_{k+1}[t]$ denotes the $t$-blowup of the clique $K_{k+1}$, which is a complete $(k+1)$-partite graph with each partite set of size $t$. 
And we write $K_k^+[t]$ for the graph obtained from $K_k[t]$ by embedding an edge to a partite set. 
 For a graph $G$ and a vertex $u\in V(G)$, 
 we write $N (u)$ for the set of neighbors of $u$. The {\it degree} $d(u)$ is the number of neighbors of $u$ in $G$. 
 Let $\delta(G)$ be the {\it minimum degree} of vertices of $G$. 
 The order of $G$, denoted by $|G|$, is the number of vertices of $G$.  
 Let $G-u$ be the subgraph obtained from $G$ by removing $u$ from $V(G)$ and removing all edges containing $u$ from $E(G)$. 
 We write $\mathbf{x}=(x_1,\ldots ,x_n)\in \mathbb{R}^n$ for the nonnegative unit  eigenvector of $G$ corresponding to $q(G)$, where $x_i$ is the $i$-th coordinate corresponding to the vertex $v_i$.

\section{Preliminaries}

\label{sec-3}

\begin{lem}[See \cite{AN2013}]\label{min}
Let $G$ be an $n$-vertex graph with $q(G)=q$ and  $\delta(G)=\delta$. 
If $\mathbf{x}=(x_{1},\ldots,x_{n})$ is a nonnegative unit eigenvector of $q(G)$, then the value $x_u=\min\{x_{1},\ldots,x_{n}\}$ satisfies 
$$x_{u}^{2}({q^{2}-2q\delta+n\delta})\leq \delta.$$
\end{lem}

\begin{lem}\label{sdf}
Let $G$ be a graph of order $n$. If $\mathbf{x}=(x_{1},\ldots,x_{n})$ is a nonnegative unit  eigenvector of $q(G)$ and 
 $u$ is a vertex for which $x_{u}=\min\{x_{1},\ldots,x_{n}\}$, then
$$\frac{q(G-u)}{n-2}\geq \frac{q(G)}{n-1}\bigg(1+\frac{1-nx_u^{2}}{(n-2)(1-x_u^{2})}\bigg)-\frac{1-nx_u^{2}}{(n-2)(1-x_u^{2})}.$$
\end{lem}

\begin{proof}
First, the Rayleigh principle implies that
\begin{align}
q(G)&=\mathbf{x}^\mathrm{T}Q(G)\mathbf{x}=\sum_{ij\in E(G)}(x_{i}+x_{j})^{2}=\sum_{ij\in E(G-u)}(x_{i}+x_{j})^{2}+\sum_{j\in N (u)}(x_u+x_{j})^{2} \notag \\
&= \sum_{ij\in E(G-u)}(x_{i}+x_{j})^{2}+d(u)x_u^{2}+2x_u\sum_{j\in N (u)}x_{j}+\sum_{j\in N (u)}x_{j}^{2} \notag \\
&\leq (1-x_u^{2})q(G-u)+d(u)x_u^{2}+2x_u\sum_{j\in N (u)}x_{j}+\sum_{j\in N (u)}x_{j}^{2}. \label{ne1}
\end{align}
By the eigenequation for the vertex $u$, we have 
$$(q(G)-d(u))x_u =\sum_{j\in N (u)}x_{j}.$$
Then we see that
\begin{displaymath}
\begin{split}
d(u)x_u^{2}+2x_u\sum_{j\in N (u)}x_{j}+\sum_{j\in N (u)}x_{j}^{2}
&= d(u)x_u^{2}+2(q(G)-d(u))x_u^{2}+\sum_{j\in N (u)}x_{j}^{2}\\
&\leq d(u)x_u^{2}+2(q(G)-d(u))x_u^{2}+1-(n-d(u))x_u^{2}\\
&= 2q(G)x_u^{2}-nx_u^{2}+1.
\end{split}
\end{displaymath}
Combining this with (\ref{ne1}), we find that
\begin{equation*} 
q(G-u)\geq q(G)\frac{1-2x_u^{2}}{1-x_u^{2}}-\frac{1-nx_u^{2}}{1-x_u^{2}}.
\end{equation*} 
Therefore, we have
\begin{align*} \notag
\frac{q(G-u)}{n-2}&\geq\frac{q(G)}{n-1}\bigg(1+\frac{1}{n-2}\bigg)\frac{1-2x_u^{2}}{1-x_u^{2}}-\frac{1-nx_u^{2}}{(n-2)(1-x_u^{2})}\\ 
&= \frac{q(G)}{n-1}\bigg(1+\frac{1-nx_u^{2}}{(n-2)(1-x_u^{2})}\bigg)-\frac{1-nx_u^{2}}{(n-2)(1-x_u^{2})}. 
\end{align*}
This completes the proof. 
\end{proof}

\subsection{Two graph-reduction algorithms}

The following result is a key ingredient in our paper. 
 The main contribution of this paper is to provide a unified approach to proving the supersaturation and stability results stated in Section \ref{sec-2} in terms of the signless Laplacian spectral radius by using Theorem \ref{h1}. 

\begin{theorem}\label{h1}
Suppose that $0<\alpha\leq2/3$, $0<\beta<1/2$, $\alpha < \gamma$,  $1/2+\beta\leq\gamma<1$, $c=6^{-3/\alpha\beta}$, $s\geq 0$ and $n$ is a sufficiently large integer. If $G$ is a graph of order $n$ with
\begin{equation}\label{s1}
q(G)>2\gamma n-s/n~~\mbox{and}~~\delta(G)\leq(\gamma-\alpha)n,
\end{equation}
then there exists an induced subgraph $H\subseteq G$ with $|H|>c n$ satisfying 
$$q(H)>2\gamma |H|~~\mbox{and}~~\delta(H)>(\gamma-\alpha)|H|.$$
\end{theorem}
\begin{proof}
 Let $n$ be a sufficiently large integer. 
 We denote $N=\lceil c n\rceil$.
We define a sequence of graphs $G_{n}, G_{n-1},\ldots, G_{t}$ by the following Algorithm \ref{thm3-3-alg-1} (We will show that this process terminates at $t> N$). 

\begin{algorithm}[H] 
\caption{Graph Reduction Algorithm Based on the Perron Eigenvector}
\label{thm3-3-alg-1}
\begin{algorithmic}[1]
\STATE Set \( G_n = G \); 
\STATE Set \( i = 0 \);
\WHILE{\( \bm{\delta(G_{n-i}) \leq (\gamma - \alpha)(n-i)} \)}
    \STATE Select a nonnegative unit eigenvector \((x_1, \ldots, x_{n-i})\) corresponding to \( q(G_{n-i}) \);
    \STATE Select a vertex \( u_{n-i} \in V(G_{n-i}) \) such that \( x_{u_{n-i}} = \min\{x_1, \ldots, x_{n-i}\} \);
    \STATE Set \( G_{n-i-1} = G_{n-i} - u_{n-i} \);
    \STATE Add $1$ to $i$;  
\ENDWHILE
\end{algorithmic}
\end{algorithm}  
 In what follows, we prove by induction that 
 for every $i=1,\ldots,n-t$, 
\begin{equation}\label{s2}
q(G_{n-i})>2\gamma (n-i)~~~\mbox{and}~~~
q(G_{n-i})\geq q(G_{n-i+1})
\left(1- \frac{1- \alpha\beta/3}{n-i} \right).
\end{equation}
To do so, we first prove that the assertion is true for the base case $i=1$. Set $q=q(G)=q(G_{n})$, $\delta=\delta(G_n)$ and $x=x_{u_{n}}=\min\{x_1,\ldots, x_n\}$. 
It follows by Lemma \ref{min} that 
$$x^{2} n\leq \frac{\delta n}{q^{2}-2q\delta+n\delta}=\frac{\delta n}{n\delta+(q-\delta)^{2}-\delta^{2}}.$$
Note that the right-hand side increases with $\delta$ and decreases with $q$ on $[\delta,+\infty)$.
In view of (\ref{s1}), 
\begin{displaymath}
\begin{split}
x^{2}n &\leq \frac{(\gamma-\alpha)n^2}{(\gamma-\alpha)n^2
+(\gamma n+\alpha n -s/n)^{2}-(\gamma-\alpha)^{2}n^2}\\
&= \frac{\gamma-\alpha}{(\gamma-\alpha) 
+4\gamma \alpha-2(\gamma+\alpha)s/n^{2}}\\
&\leq \frac{\gamma -\alpha}{\gamma-\alpha
+3\gamma \alpha} \\
&= 1- \frac{3\gamma \alpha}{\gamma-\alpha
+3\gamma \alpha}\\
&\leq  1-\alpha,
\end{split}
\end{displaymath}
where the last inequality follows from $0<\alpha\leq2/3$.

On the other hand, by $1/2+\beta\leq\gamma<1$, we have
$$q(G_{n})=q(G)>2\gamma n-s/n>(1+\beta)n.$$
Note that $0<\beta<1/2$ and $x^{2}n\leq  1-\alpha$. 
By Lemma \ref{sdf}, we see that
\begin{align*}
\begin{split}
\frac{q(G_{n-1})}{n-2}&\geq \frac{q(G_{n})}{n-1}\bigg(1+\frac{1-nx^{2}}{(n-2)(1-x^{2})}\bigg)-\frac{1-nx^{2}}{(n-2)(1-x^{2})}\\
&\geq \frac{q(G_{n})}{n-1}\bigg(1+\frac{\beta(1-nx^{2})}{(1+\beta)(n-2)(1-x^{2})}\bigg)\\
&\geq \frac{q(G_{n})}{n-1}\bigg(1+\frac{2\beta(1-nx^{2})}{3(n-1)}\bigg)\\
&\geq \frac{q(G_{n})}{n-1}\bigg(1+\frac{2\alpha\beta}{3(n-1)}\bigg),
\end{split}
\end{align*}
that is
$$q(G_{n-1})\geq q(G_{n})\bigg(1-\frac{1}{n-1}\bigg)\bigg(1+\frac{2\alpha\beta}{3(n-1)}\bigg)\geq
q(G_{n})\left( 
1- \frac{1-\alpha\beta/3}{n-1} \right).$$
Since $q(G)>2\gamma n-s/n$, we get 
$$q(G_{n-1})> \left(2\gamma n- \frac{s}{n} \right)
\left(1- \frac{1-\alpha\beta/3}{n-1} \right)>2\gamma (n-1).$$
Assume that  (\ref{s2}) holds for $1\leq i\leq n-(t-1)$, that is, 
$$q(G_{n-i})>2\gamma (n-i)>2\gamma (n-i)-s/(n-i),$$
and $$\delta(G_{n-i})\leq(\gamma-\alpha)(n-i).$$
We can repeat the above process and prove the claim for $i+1$, as needed.

 Next, we show that $t>N$. Suppose on the contrary that the sequence reaches the graph $G_{N}$.  
 We will show that $q(G_N)\ge 2N$, which leads to a contradiction since $q(G_N)\le 2\Delta (G_N)\le 2(N-1)$.   
 It is easy to check that 
 $\ln (1-ax) >- (ax +x^2)$ for $0<x< 1/2$ and $0<a<1$. Moreover, we need to use the inequality: $\frac{1}{1+y}< \ln \frac{y+1}{y} < \frac{1}{y}$ for any $y>0$.  
 By applying (\ref{s2}), we have
\begin{displaymath}
\begin{split}
q(G_{N})
&\geq q(G_{N+1})\left( 
1- \frac{1-\alpha\beta/3}{N} \right)\\
&\geq q(G_{n})\prod_{i=N+1}^{n}\left(1- 
\frac{1-\alpha \beta/3}{i-1} \right)\\
&\geq q(G_{n}) 
\exp\left(-\sum_{i=N+1}^{n}
\left( \frac{1-\alpha \beta/3}{i-1} + 
\frac{1}{(i-1)^{2}} \right)\right)\\ 
&\geq q(G_{n})\exp\left(-(1-\alpha\beta/3) \ln \frac{n}{N-1}-1\right)\\
&\geq (2\gamma n-s/n)\left(\frac{n}{N-1}\right)^{-(1-\alpha\beta/3)}e^{-1}\\
&\geq e^{-1}n^{\alpha\beta/3}N^{1-\alpha\beta/3}\\  
&\geq 2N,
\end{split}
\end{displaymath}
where the second to last inequality follows from
$2\gamma n-s/n\geq(1+\beta) n$ and $e^{-1}(1+\beta)(N-1)^{1-\alpha\beta/3}\geq e^{-1}N^{1-\alpha\beta/3}$; and the last inequality follows from $c=6^{-3/\alpha\beta}$ and $N=\lceil c n\rceil\leq (2e)^{-3/\alpha\beta}n$. 
So we conclude that $q(G_N)\ge 2N$. 
This is a contradiction. 

Consequently, the sequence of graphs will terminate at some $G_t$ where $t>N$. From the construction,
we have 
$q(G_{t})>2\gamma t$ and $\delta(G_{t})>(\gamma-\alpha)t$, which completes the proof. 
\end{proof}

By modifying the proof of Theorem \ref{h1}, we can prove the following variant of Nikiforov \cite[Theorem 5]{Nik2008-LAA}. 
Since the argument is similar, we postpone the detailed proof to Appendix \ref{sec-A}. 

\begin{theorem}\label{h1-two-cases}
Suppose that $0<\alpha\leq2/3$, $0<\beta<1$, $\alpha <\gamma$, $31/48\leq\gamma<1$,  $s\geq 0$ and $n$ is a sufficiently large integer. If $G$ is a graph of order $n$ with
$$q(G)>2\gamma n-s/n~~\mbox{and}~~\delta(G)\leq(\gamma-\alpha)n,$$
then there is an induced subgraph $H\subseteq G$ with $|H|>(1-\beta) n$ and  satisfying
one of the following:
\begin{itemize}
\item[\rm (i)]
$q(H)>2\gamma(1+\alpha\beta/8)|H|;$
\item[\rm (ii)]
$q(H)>2\gamma |H|~~\mbox{and}~~\delta(H)>(\gamma-\alpha)|H|.$
\end{itemize}
\end{theorem}

{
\subsection{Some known results for dense graphs}

In this section, we introduce some preliminary results for graphs under the density condition or the minimum degree condition. These results will serve as crucial ingredients in our proofs. 
The following result is the well-known supersaturation lemma due to Erd\H{o}s and Simonovits \cite{ES1983-super}. 

\begin{lem}[Erd\H{o}s--Simonovits \cite{ES1983-super}]  \label{lem-super-cliques}
If $k\ge 1$, $\varepsilon >0$ and 
 $G$ is an $n$-vertex graph with 
$e(G) \ge ( 1-\frac{1}{k}+\varepsilon ) \frac{n^2}{2}$, 
then $G$ contains at least $\Omega_{k,\varepsilon} (n^{k+1})$ copies of $K_{k+1}$. 
\end{lem}

To prove Theorem \ref{lsES}, 
we need to use a result of Bollob\'as and Erd\H{o}s  \cite{BE1973}.  

\begin{lem}[Bollob\'{a}s--Erd\H{o}s \cite{BE1973}] \label{lem-BE1973} 
If $k\ge 1, \varepsilon >0$ and $G$ is an $n$-vertex graph with 
$ e(G)\ge (1-\frac{1}{k}+\varepsilon ) \frac{n^2}{2}$, then $G$  
contains a copy of $K_{k+1}[t]$ with  $t=\Omega_{k,\varepsilon}(\log n)$. 
\end{lem}

The following lemma of Nikiforov \cite{N2010} is needed in the proof of Theorem \ref{thm-h3}. 

\begin{lem}[Nikiforov \cite{N2010}]\label{h2}
Let $k\geq2$ and $G$ be a graph of order $n$. If $G$ contains a copy of $K_{k+1}$ and
$\delta(G)>(1- \frac{1}{k}- \frac{1}{k^{4}})n$, then $G$ contains a copy of 
$K_{k}^{+}[t]$ with $t=\Omega_k(\log n)$.
\end{lem}

To prove Theorem \ref{thm-q-js}, we need to use the following result \cite[Lemma 6]{BN2008}.  

\begin{lem}[Bollob\'{a}s--Nikiforov \cite{BN2008}]  \label{lem-lem6}
 Let $k\ge 2$ and $G$ be a graph of order $n$. 
If $G$ contains a copy of $K_{k+1}$ and $\delta (G)>(1-\frac{1}{k}-\frac{1}{k^4})n$, 
then $js_{k+1}(G) > {n^{k-1}}/{k^{k+3}}$. 
\end{lem}

In 2023, Fox and Wigderson \cite[Lemma 2.3]{FW2023} provided an improvement of Lemma \ref{lem-lem6} by proving that 
if $k\ge 2, \alpha >0$ and $G$ is an $n$-vertex graph with $ \delta (G) \ge (1- \frac{2}{2k-1} + \alpha )n$, 
then every copy of $K_{k+1}$ in $G$ 
contains an edge that lies in at least $C_k \alpha n^{k-1}$ copies of $K_{k+1}$ 
for some constant $C_k>0$ depending only on $k$. Later, Fox et al. \cite[Lemma A.1]{FHW2021} showed that $C_k\ge (8k+8)^{-(k-1)}$.

To prove Theorem \ref{sst}, we need to use the 
Erd\H{o}s--Simonovits stability. 

\begin{lem}[Erd\H{o}s--Simonovits \cite{Erd1966Sta1,Erd1966Sta2,S1968}]\label{est}
Let $F$ be a graph with $\chi(F)=k+1\geq3$. For every $\varepsilon>0$, there exist $\sigma>0$ and $n_{0}$ such that
if $G$ is an $F$-free graph on $n\geq n_{0}$ vertices with $e(G)\geq(1- \frac{1}{k} -\sigma)\frac{n^{2}}{2}$, then $G$ can be obtained from $T_{n,k}$ by adding and deleting at most $\varepsilon n^{2}$ edges.
\end{lem}
}

\section{Proof of Theorems \ref{thm-1-2} and \ref{lsES}} 

\label{sec-4}

We now prove Theorem \ref{thm-1-2} by using Theorem \ref{h1} and Lemma \ref{lem-super-cliques}. 
  
\begin{proof}[{\bf Proof of Theorem \ref{thm-1-2}}]
Suppose that $G$ is an $n$-vertex graph with 
{$q(G)\ge (1- \frac{1}{k}+ {\varepsilon})2n$}. 
We start with the proof in two cases. 
    If $\delta (G) > (1- \frac{1}{k}+ \frac{\varepsilon}{2})n$, 
    then 
    \[ e(G)\geq \frac{n}{2}\delta(G)> 
\left(1-\frac{1}{k}+\frac{\varepsilon}{2} \right) \frac{n^2}{2}. \]  
By Lemma \ref{lem-super-cliques}, 
we can find $\Omega_{k,\varepsilon}(n^{k+1})$ copies of $K_{k+1}$ in $G$.  
Now, we assume that $\delta(G)\leq (1 - \frac{1}{k} + \frac{\varepsilon}{2})n$.  In this case, we shall apply 
Theorem \ref{h1} to obtain a subgraph with larger minimum degree. 
We denote $\gamma=1-1/k + \varepsilon \geq 1/2+ \varepsilon$, 
$\alpha= \varepsilon /2$, $\beta= \varepsilon $ and $s=0$. Applying Theorem \ref{h1}, it follows that $G$ contains an induced subgraph $H$ on $|H|>6^{-6/\varepsilon^{2}} n$ vertices satisfying
\[ q(H)>2\gamma |H| \] 
and 
\[ \delta(H)>(\gamma-\alpha)|H|= 
\left(1-\frac{1}{k}+\frac{\varepsilon}{2} \right)|H|. \]  
Therefore, we have $e(H) \geq \frac{|H|}{2}\delta(H)> (1-\frac{1}{k}+\frac{\varepsilon}{2} )\frac{|H|^{2}}{2}$. 
By Lemma \ref{lem-super-cliques} again, $H$ contains $\Omega_{k,\varepsilon} (|H|^{k+1})$ copies of $K_{k+1}$. 
Note that $|H|> 6^{-6/\varepsilon^{2}} n$. Then $G$ contains $\Omega_{k,\varepsilon}(n^{k+1})$ copies of $K_{k+1}$. 
\end{proof}

In what follows, we give the proof of Theorem \ref{lsES} by using Lemma \ref{lem-BE1973}.

\begin{proof}[{\bf Proof of Theorem \ref{lsES}}] 
If $G$ has the minimum degree $\delta(G)> (1-\frac{1}{k}+\frac{\varepsilon}{2} )n$,
then 
\[ e(G)\geq \frac{n}{2}\delta(G)> 
\left(1-\frac{1}{k}+\frac{\varepsilon}{2} \right) \frac{n^2}{2}. \]  
By Lemma \ref{lem-BE1973}, 
$G$ contains a copy of $K_{k+1}[t]$ with $t=\Omega_{k,\varepsilon} (\log n)$. 
In the following, we assume that $\delta(G)\leq (1 - \frac{1}{k} + \frac{\varepsilon}{2})n$.  In this case, we apply 
Theorem \ref{h1} to obtain a dense subgraph. 
We denote $\gamma=1-1/k + \varepsilon \geq 1/2+ \varepsilon$, 
$\alpha= \varepsilon /2$, $\beta= \varepsilon $ and $s=0$. Applying Theorem \ref{h1}, it follows that $G$ contains an induced subgraph $H$ on $|H|>6^{-6/\varepsilon^{2}} n$ vertices satisfying
\[  q(H)>2\gamma |H| \] 
and 
\[  \delta(H)>(\gamma-\alpha)|H|= 
\left(1-\frac{1}{k}+\frac{\varepsilon}{2} \right)|H|. \]  
Therefore, we have $e(H) \geq \frac{|H|}{2}\delta(H)> (1-\frac{1}{k}+\frac{\varepsilon}{2} )\frac{|H|^{2}}{2}$. 
By Lemma \ref{lem-BE1973} again, we find that $H$ contains a copy of $K_{k+1}[t]$ with $t=\Omega_{k,\varepsilon}(\log |H|) = \Omega_{k,\varepsilon} (\log n)$, as desired.  
\end{proof}

\noindent 
{\bf Remark.}
We point out here that Theorem \ref{lsES} can also be deduced from Theorem \ref{thm-1-2} together with a theorem of Nikiforov \cite{Niki2008blms}, which says that every $n$-vertex graph with $\varepsilon n^{k+1}$ copies of $K_{k+1}$ contains a copy of $K_{k+1}[t]$ of size $t=\Omega_{k,\varepsilon}(\log n)$. 
A recent result of Brada\v{c} et al. \cite{BLWX2024} states that every incomparable graph of order $n$ with $\varepsilon n^{k+1}$ copies of $K_{k+1}$ contains a copy of $K_{k+1}[t]$ with $t=\Omega_{k,\varepsilon}(\frac{n}{\log n})$. Combining this result with Theorem \ref{thm-1-2},  
we obtain that any incomparable graph of order $n$ with $q(G)\ge (1- \frac{1}{k} + \varepsilon)2n$ contains a copy of $K_{k+1}[t]$ with $t=\Omega_{k,\varepsilon}(\frac{n}{\log n})$.

\medskip 
The following example shows that the bounds 
in both Theorems \ref{sES} and \ref{lsES} are optimal. 

\begin{example}
    \label{exam-2-4}
    Given $k\ge 2$, $0< \varepsilon < 1/(2k^2) $ and $n\ge 2k/\varepsilon$,  
there exists an $n$-vertex graph $G$ with 
$q(G)\ge (1- \frac{1}{k} + \varepsilon)2n$, 
but $G$ does not contain a copy of $K_{k+1}[t]$ 
with $t=\big\lceil\frac{2\log (n/k)}{\log (1/2\varepsilon k^2)}\big\rceil$.  
\end{example} 

\begin{proof} 
The proof uses the standard probabilistic method.  
We start with the $k$-partite Tur\'{a}n graph 
$T_{n,k}$, where $k$ divides $n$. 
Let $U$ be a partite set of $T_{n,k}$ of size $n/k$.  
Our strategy is to add $\frac{1}{2}\varepsilon n^2$ extra edges within $U$ to obtain a new graph $G$ such that $G[U]$ does not contain a copy of $K_{t,t}$. 

We choose edges inside $U$ independently with probability $p$, where the value of $p$ will be determined later. 
Let $X$ be the number of edges and 
$Y$ be the number of copies of $K_{t,t}$ created in $U$.  
Next, we are going to compute the expectations  $\mathbb{E}[X]$ and 
$\mathbb{E}[Y]$.  
Setting $p:=2\varepsilon k^2 <1$, we have 
\[   \mathbb{E}[X] = p {n/k \choose 2} >  \frac{1}{2}\varepsilon n^2 + 1, \] 
where the inequality holds by $n\ge 2k/ \varepsilon$. 
On the other hand, we would like to require  
\[  \mathbb{E}[Y] = \frac{1}{2} {n/k \choose t} {n/k -t \choose t} p^{t^2} 
< \frac{1}{2}\left(\frac{n}{k} \right)^{2t} (2\varepsilon k^2)^{t^2} < 1. \]
By direct computations, it suffices to require $2t \log \frac{n}{k} + 
t^2 \log (2 \varepsilon k^2) \le 0$, which can be achieved by 
setting $t:= \big\lceil\frac{2\log ({n}/{k})}{ \log ({1}/{2\varepsilon k^2})}\big\rceil$. 
By the linearity of expectation, 
we get 
\[ \mathbb{E}[X-Y]  =\mathbb{E}[X]-\mathbb{E}[Y] 
 > \frac{1}{2}\varepsilon n^2 . \]   
 Therefore, there is a distribution of edges within $U$ such that $X-Y> \frac{1}{2}\varepsilon n^2$. 
 Removing an edge from each copy of $K_{t,t}$, we get a $K_{t,t}$-free graph of order $n/k$ on $U$ with more than $\frac{1}{2}\varepsilon n^2$ edges. 
 Note that $U$ is a partite set of $T_{n,k}$. 
Hence, the resulting graph $G$ is a $K_{k+1}[t]$-free graph of order $n$. 
Since $e(G) > (1- \frac{1}{k}) \frac{n^2}{2} + \frac{1}{2}\varepsilon n^2$, 
it follows that $q(G)\ge \frac{4e(G)}{n} > (1-\frac{1}{k} + \varepsilon )2n $. 
\end{proof}

\section{Proof of Theorems \ref{thm-h3}, \ref{thm-q-js} and \ref{thm-genelized-books}}

\label{sec-5}

The primary advantage of Lemma \ref{h2} lies in deriving the bound on the size $t=\Omega_k(\log n)$ from two simpler conditions: the presence of a clique $K_{k+1}$ and a sufficiently large minimum degree. While these conditions may not necessarily hold in the original graph $G$, Theorem \ref{h1} ensures the existence of a large subgraph $H$ where they are satisfied. By applying Lemma \ref{h2} to the subgraph $H$, we thereby obtain the desired bound on the partite sets.

\begin{proof}[{\bf Proof of Theorem \ref{thm-h3}}]
Let $G$ be a graph of order $n$ with  $q(G)\ge q( T_{n,k})$ and $G\neq T_{n,k}$. 
By (\ref{eq-q-Turan}), we know that $G$ contains a copy of $K_{k+1}$. If $\delta(G)>(1- 
\frac{1}{k} - \frac{1}{k^{4}})n$, 
then by Lemma \ref{h2}, we obtain that $G$ contains a copy of $K_{k}^{+}[t]$ with $t=\Omega_k(\log n)$. 
In what follows, we consider the case 
$\delta(G)\leq (1- \frac{1}{k} - \frac{1}{k^4})n$. 
 Using the fact that
$e( T_{n,k})\geq(1- \frac{1}{k})\frac{n^{2}}{2} - \frac{k}{8}$, we obtain 
\[  q(T_{n,k})\ge \frac{4e(T_{n,k})}{n}
\ge 2\left(1- \frac{1}{k} \right)n- \frac{k}{2n}. \] 
We denote $\gamma=1-1/k\geq 2/3$, $\alpha=1/k^{4}$, $\beta=1/6$ and $s=k/2$. 
Then 
$$q(G)>q( T_{n,k})\ge 2\gamma n-s/n.$$ 
Applying Theorem \ref{h1}, we know that $G$ contains an induced subgraph $H$ with  $|H|>6^{-18k^{4}} n$ vertices such that   
\[ \delta(H)>(\gamma-\alpha)|H| = 
\left(1- \frac{1}{k} - \frac{1}{k^4} \right)|H| \] 
and 
\[ q(H)>2\gamma |H|=2\left(1- \frac{1}{k}\right)|H|. \]  
By (\ref{eq-q-Turan}),  
we know that $H$ must contain a copy of $K_{k+1}$. 
By Lemma \ref{h2} again, we get that $H$ contains a copy of 
$K_{k}^{+}[t]$ with $t=\Omega_k( \log |H|)$.  Note that $|H|= \Omega_k(n)$. Then 
$G$ contains a copy of $K_k^+[t]$ with $t=\Omega_k (\log n)$. 
\end{proof}

The following example shows that 
the bound in Theorem \ref{thm-h3} is tight up to a constant factor. 

\begin{example} \label{exam-2-6}
For every integers $k\ge 2$ and $n>3k$, 
    there exists a graph $H$ on $n$ vertices 
with $q(H)> q(T_{n,k})$, while  
$H$ does not contain a copy of $K_{k}^+[t]$ where $t= \big\lceil\frac{2\log (2n/k)}{\log (3/2)} \big\rceil$.  
\end{example}

\begin{proof}
(Using the probabilistic method).  
Our strategy is to replace the complete bipartite graph formed between two vertex classes of the Tur\'{a}n graph $T_{n,k}$ 	
  with a graph with the same density (not necessarily bipartite) that does not contain a large $K_{t,t}$. 
We may assume that $k$ divides $n$. 
We consider the random graph $G(2n/k,p)$, a graph on $2n/k$ vertices, 
where each edge is presented independently with probability $p:=2/3$. 
Then we have 
\[ \mathbb{E}[e(G)] = p {2n/k \choose 2} > 
\left(\frac{n}{k} \right)^2 +1. \]  
Let $t:=\big\lceil\frac{2\log (2n/k)}{\log (3/2)}\big\rceil$ and $\# K_{t,t}$ be the number of copies of $K_{t,t}$ in $G$. Then 
\[  \mathbb{E}[\# K_{t,t}] = \frac{1}{2} {2n/k \choose t} {2n/k -t \choose t} p^{t^2} 
<  \frac{1}{2} \left(\frac{2n}{k} \right)^{2t} 
\left(\frac{2}{3} \right)^{t^2} < 1. \]  
The linearity of expectation implies 
\[ \mathbb{E}[e(G)- \#K_{t,t}] = 
\mathbb{E}[e(G)] - \mathbb{E}[\#K_{t,t}] 
 > \left(\frac{n}{k} \right)^2.\]   
 Thus, there exists 
a graph $G$ of order $2n/k$ in which $e(G)- \#K_{t,t} > (\frac{n}{k})^2$. 
By removing one edge 
from each copy of $K_{t,t}$ in $G$, we get a $K_{t,t}$-free subgraph $G'$ of order $2n/k$ 
with $e(G')> (\frac{n}{k})^2$. Now, 
we consider the graph $H=G' \vee T_{n-{2n}/{k},k-2}$, 
which is obtained by joining every vertex of $G'$ to every vertex of the $(k-2)$-partite Tur\'{a}n graph $T_{n-{2n}/{k},k-2}$. 
Clearly, $H$ has order $n$ and $e(H)> (1- \frac{1}{k})\frac{n^2}{2}$, which yields $q(G)> (1- \frac{1}{k})2n \geq q(T_{n,k})$. 
Since $G'$ has no copy of $K_{t,t}$ with $t=\big\lceil\frac{2\log (2n/k)}{\log (3/2)}\big\rceil$, we know that $H$ contains no copy of $K_k[t]$ and no copy of $K_k^+[t]$ as well. 
\end{proof}

In what follows, we give the proof of Theorem \ref{thm-q-js}. 

\begin{proof}[{\bf Proof of Theorem \ref{thm-q-js}}]
    Let $G$ be a graph of order $n$ with  $q(G)\ge q( T_{n,k})$ and $G\neq T_{n,k}$. 
By (\ref{eq-q-Turan}), we know that $G$ contains a copy of $K_{k+1}$. If $\delta(G)>(1- 
\frac{1}{k} - \frac{1}{k^{4}})n$, 
then by Lemma \ref{lem-lem6}, we get 
\[ js_{k+1}(G)=\Omega_k(n^{k-1}).\]  
Now, we consider the case 
$\delta(G)\leq (1- \frac{1}{k} - \frac{1}{k^4})n$. 
It is known that $q(T_{n,k}) \ge 2(1- \frac{1}{k})n- \frac{k}{2n}
$. We denote $\gamma=1-1/k\geq 2/3$, $\alpha=1/k^{4}$, $\beta=1/6$ and $s=k/2$. 
Then 
$$q(G)>q( T_{n,k})\ge 2\gamma n-s/n.$$ 
Applying Theorem \ref{h1}, we obtain that $G$ contains an induced subgraph $H$ with $|H|>6^{-18k^{4}} n$ vertices such that   
\[ \delta(H)>(\gamma-\alpha)|H| = 
\left(1- \frac{1}{k} - \frac{1}{k^4} \right)|H| \] 
and 
\[ q(H)>2\gamma |H|=2 \left(1- \frac{1}{k}\right)|H|. \]  
By (\ref{eq-q-Turan}),  
we know that $H$ contains a copy of $K_{k+1}$. 
By Lemma \ref{lem-lem6} again, we find that 
$js_{k+1}(H)=\Omega_k(|H|^{k-1})$. 
Since $|H|=\Omega_k(n)$, we know that 
$js_{k+1}(G)=\Omega_k (n^{k-1})$, as desired. 
\end{proof}

Next, we show that 
Theorem \ref{thm-genelized-books} is a direct consequence of Theorem \ref{thm-q-js}. 

\begin{proof}[{\bf Proof of Theorem \ref{thm-genelized-books}}]
 Applying Theorem \ref{thm-q-js}, 
we know that $G$ contains a large $(k+1)$-joint of size $js_{k+1}(G)=\Omega_k(n^{k-1})$. 
In other words, there exists an edge $\{u,v\}\in E(G)$ such that $G$ contains $\Omega_k(n^{k-1})$ copies of $K_{k+1}$ sharing $\{u,v\}$. 
 We consider the induced subgraph $G'=G[N(u)\cap N(v)]$. 
 Each copy of $K_{k+1}$ in the joint of $G$  corresponds to a copy of $K_{k-1}$ of $G'$. 
 Thus, it follows that 
 $G'$ contains $\Omega_k(n^{k-1})$ copies of 
 $K_{k-1}$. Clearly, $G'$ has at most $O_k(n^{k-2})$ copies of $K_{k-2}$. 
 By the pigeonhole principle, there are 
 $\Omega_k(n)$ copies of $K_{k-1}$ in $G'$ that share a common $K_{k-2}$. 
 Combining with the edge $\{u,v\}$, 
 we see that $G$ contains $\Omega_k(n)$ copies of $K_{k+1}$ that share a common $K_k$. In other words, $G$ contains a copy of $B_{k,t}$ 
 with $t=\Omega_k(n)$, as needed. 
\end{proof}

\section{Proof of Theorem \ref{sst}}

\label{sec-6}

Now, we prove the stability theorem for the signless Laplacian spectral radius. 

\begin{proof}[{\bf Proof of Theorem \ref{sst}}]
For any number $\varepsilon>0$, we know from Lemma \ref{est} that there exist $\sigma_{1}>0$ and $n_{1}$ such that if
$G$ is an $F$-free graph on $n\ge n_1$ vertices with 
$e(G)\geq(1- \frac{1}{k} -\sigma_{1})\frac{n^{2}}{2}$, then
$G$ can be obtained from $T_{n,k}$ by adding and deleting at most $\varepsilon n^2$ edges.  

We denote $t:=1- \frac{1}{k}\geq \frac{2}{3}$, 
$ \sigma := \frac{\alpha\beta t}{16+2\alpha\beta}$ and $\gamma :=t-\sigma$, where $\alpha>0$ and $\beta>0$ are sufficiently small numbers such that $\sigma < 0.001$, $ \frac{31}{48}\le \gamma <1$ and  
\begin{equation}  \label{2-1}
    (\gamma -\alpha)(1-\beta)^{2} 
  = (t-\sigma - \alpha)(1- \beta)^2 \geq t-\sigma_{1}. 
\end{equation}  
Assume that $G$ is an $n$-vertex $F$-free graph with $q(G)\geq 2\gamma n$.  
If $\delta(G)>(\gamma-\alpha)n$, then
\[ e(G)\geq \frac{n\delta(G)}{2}> (\gamma-\alpha)\frac{n^{2}}{2}\geq (t-\sigma_{1}) \frac{n^{2}}{2}, \] 
where the last inequality holds by (\ref{2-1}).  
Applying  Lemma \ref{est}, we know that $G$ can be obtained from $T_{n,k}$ by adding and deleting at most $\varepsilon n^2$ edges. 
Next, we assume that $\delta(G)\leq(\gamma-\alpha)n$. Recall that $q(G)\geq2\gamma n$. Setting $s=0$ in Theorem \ref{h1-two-cases},  
we obtain that for sufficiently large $n$, there exists an induced subgraph $H\subseteq G$ with $|H|>(1-\beta) n$ and  satisfying
one of the following statements:
\begin{itemize}
\item[(i)]
$q(H)>2\gamma(1+\alpha\beta/8)|H|;$
\item[(ii)]
$q(H)>2\gamma |H|~~\mbox{and}~~\delta(H)>(\gamma-\alpha)|H|.$
\end{itemize}
Assume that (i) holds. It follows from Theorem \ref{eq-q-ESS} that for sufficiently large $n$, 
$$\frac{\mathrm{ex}_{q}(n,F)}{n}<2t \bigg(1+\frac{\alpha\beta }{16}\bigg).$$
 Note that $H$ is $F$-free and $|H|> (1- \beta )n$. Then we have 
$$2(t-\sigma)\left( 1+ \frac{\alpha\beta}{8} \right)=2\gamma 
\left(1+ \frac{\alpha\beta}{8} \right)\leq\frac{q(H)}{|H|}\leq
\frac{\mathrm{ex}_{q}(|H|,F)}{|H|}<2t \bigg(1+\frac{\alpha\beta}{16}\bigg).$$
Simplifying the above inequality gives  $\sigma>\alpha\beta t/(16+2\alpha\beta)$, which contradicts with the assumption on $\sigma$.  
This means that the statement in (i) does not happen. 
Now, we assume that (ii) holds.
By $|H|>(1-\beta)n$ and $\delta (H) > (\gamma - \alpha )|H|$, we have
$$e(G)\geq e(H)\geq \frac{|H|\delta(H)}{2}> (\gamma-\alpha)(1-\beta)^{2} \frac{n^{2}}{2}\geq (t-\sigma_{1}) \frac{n^{2}}{2},$$
where the last inequality holds by (\ref{2-1}). Using Lemma \ref{est}, we obtain that $G$ can be obtained from $T_{n,k}$ by adding and deleting at most $\varepsilon n^2$ edges. This completes the proof. 
\end{proof}

\section{Proof of Theorem \ref{thm-2-14}}

\label{sec-7}

Note that every graph $F$ with $\chi (F)=k+1$ can be contained in the blowup $K_{k+1}[t]$ with a large integer $t$. To prove Theorem \ref{thm-2-14}, 
we show a stronger result as follows. 

\begin{theorem} \label{thm-degree-1}
    Let $k\ge 2$ and $\varepsilon >0$. 
    If $G$ is a graph on $n$ vertices with $m$ edges such that 
    \[ \sum_{v\in V(G)}d^2(v) \ge 
    2\left(1- \frac{1}{k} +\varepsilon \right) mn, \] 
    then $G$ contains $\Omega_{k,\varepsilon}(n^{k+1})$ copies of $K_{k+1}$, and a copy of $K_{k+1}[t]$ with $t=\Omega_{k,\varepsilon}(\log n)$. 
\end{theorem}

\begin{proof}
In our argument, we need to use a lower bound on $q(G)$.
This bound can be found in \cite[Theorem 2.1]{LL2009-Zagreb} or \cite[Lemma 3]{Zhou2010}, and it states that
\begin{equation} \label{eq-signless}
  q (G) \ge \frac{1}{m} \sum_{v\in V(G)} d^2(v),
  \end{equation}
where the equality holds if and only if  $d(u)+d(v)$ are equal for
any $uv \in E(G)$.
Assume that $G$ is a graph with $\sum_{v\in V(G)}d^2(v) \ge 2 (1- \frac{1}{k} + \varepsilon)mn$. 
It follows from (\ref{eq-signless}) that 
$q(G)\ge (1- \frac{1}{k} + \varepsilon)2n$. 
Thus, the desired results are followed by applying Theorems \ref{thm-1-2} and \ref{lsES}. 
\end{proof}

Following a similar framework of the above proof, 
we can prove the following result 
 by combining (\ref{eq-signless}) with Theorems \ref{thm-h3}, \ref{thm-q-js} and \ref{thm-genelized-books}. 
This extends the Nikiforov--Rousseau result (\ref{thm-NR-degree}).

\begin{theorem} \label{thm-degree-2}
    If $k\ge 3$ is an integer and $G$ is a graph on $n$ vertices with $m$ edges such that  
      \[ \sum_{v\in V(G)}d^2(v) > 
    2\left(1- \frac{1}{k} \right) mn,\]  
    then $G$ contains a copy of $K_k^+[t]$ with 
    $t=\Omega_k (\log n)$, a joint of size $js_{k+1}(G)=\Omega_k(n^{k-1})$, and a generalized book $B_{k,t}$ with $t=\Omega_k(n)$. Moreover, the bounds are optimal up to a constant factor. 
\end{theorem}

Observe that Theorem \ref{thm-2-14} follows immediately by Theorem \ref{thm-degree-1}.   
In what follows, we give an alternative proof of Theorem \ref{thm-2-14}. 
This proof relies on a recent result of Li, Liu and Zhang \cite{LLZ2025-stability} for the adjacency spectral radius, instead of the signless Laplacian spectral radius.  

\begin{proof}[{\bf Alternative proof of Theorem \ref{thm-2-14}}]
The well-known Hofmeistar inequality asserts that
\begin{equation*}
  \lambda^2 (G) \ge \frac{1}{n} \sum_{v\in V(G)} d^2(v),
  \end{equation*}
where the equality holds if and only if $G$ is either regular or bipartite semi-regular. 

Let $F$ be a graph on $s$ vertices with chromatic number $\chi (F)=k+1\ge 3$. 
We may assume that $G$ is an $F$-free graph that achieves the maximum degree power. 
Let $t:=\lfloor \frac{n}{s-1} \rfloor$ and $G'$ be a graph with $V(G')=V_0\cup V_1\cup \cdots V_t$, where $|V_1|=\cdots = |V_t|=s-1$ and $0\le |V_0|\le s-1$, and $G'[V_i]$ is a clique of $G'$ for every $0\le i\le t$. Each component $G'[V_i]$ has at most $s-1$ vertices, so $G'$ is $F$-free. 
Note that $\sum_{v\in V(G')}d^2(v)\ge \lfloor \frac{n}{s-1}\rfloor \cdot (s-1) \cdot (s-2)^2 = \Omega(n)$. 
By the maximality of $G$, we have
$\Omega (n) \le \sum_{v\in V(G)}d^2(v) \le \bigl( 
\sum_{v\in V(G)} d(v) \bigr)^2= 4m^2$. 
So $m$ can be large enough as $n\to \infty$. 

Recently, Li, Liu and Zhang  \cite{LLZ2025-stability} investigated the Brualdi--Hoffman--Tur\'{a}n type problem. 
It was shown in \cite{LLZ2025-stability} that if $m$ is sufficiently large and 
$G$ is an $F$-free graph with $m$ edges, then
\begin{equation*}
  \lambda^2(G) \le \left( 1-\frac{1}{k} + o(1) \right)2m,
  \end{equation*}
Therefore,
it follows that $\sum_{v\in V(G)} d^2(v) \le 2\left(1-\frac{1}{k} + o(1) \right)mn$, as expected.
\end{proof}

\section{Proofs of Theorems \ref{lhslES} and \ref{ESSL}}

\label{sec-8}

Let $H$ be a linear $r$-graph.  The  {\it $s$-shadow} of $H$, denoted by $\partial_{s}H$, is the $s$-graph with vertex set $V(\partial_{s}H)=V(H)$ and edge set $E(\partial_{s}H)=\{\{v_{1},\ldots,v_{s}\}: \{v_{1},\ldots,v_{s}\}\subseteq e\in E(H)\}$.
 In other words, $\partial_{s}H$ is obtained from $H$ by replacing each edge $e$ with a complete $s$-graph on the vertices of $e$. For simplicity, 
 we write $\partial H$ for $\partial_{2}H$. 
 In what follows, we establish the connections of the signless Laplacian spectral radius between a linear hypergraph $H$ and its $s$-shadow graph $\partial_s H$.

\begin{lem}\label{gj}
Let $H$ be a linear $r$-graph. Then
$$q(H)\leq\frac{1}{\binom{r-1}{s-1}}q(\partial_{s} H).$$
In particular, we have $(r-1)q(H)\leq q(\partial H)$.
\end{lem}

\begin{proof}
Let $\mathbf{x}$ be a nonnegative eigenvector corresponding to $q(H)$. 
Then  $\sum_{i\in V(H)} x_i^r =1$ and 
\[  q(H)=\sum_{i\in V(H)}d_{H}(i)x_{i}^{r}+r\sum_{e\in E(H)}x^{e}. \]
For each fixed $i\in V(H)$ and $e\in E(H)$ with $i\in e$, 
there are ${r-1 \choose s-1}$ subsets of $e$ with size $s$ and containing $i$. 
Then $d_H(i)= \frac{1}{{r-1 \choose s-1}}d_{\partial_s H}(i)$. Therefore, we have 
\begin{align}\label{gj1}
\begin{split}
q(H) = \frac{1}{\binom{r-1}{s-1}}\sum_{i\in V(\partial_{s}H)}d_{\partial_{s}H}(i)(x_{i}^{r/s})^{s}
+r\sum_{e\in E(H)}\prod_{\{i_{1},\ldots,i_{s}\}\subset e} 
(x_{i_{1}}\cdots x_{i_{s}})^{\frac{1}{\binom{r-1}{s-1}}}.
 \end{split}
\end{align} 
Using the AM-GM inequality and the Power-Mean inequality, we obtain 
\begin{align}
r\sum_{e\in E(H)}\prod_{\{i_{1},\ldots,i_{s}\}\subset e}(x_{i_{1}}\cdots x_{i_{s}})^{\frac{1}{\binom{r-1}{s-1}}}
&\leq  r\sum_{e\in E(H)}\bigg(\frac{1}{\binom{r}{s}}\sum_{\{i_{1},\ldots,i_{s}\}\subset e}(x_{i_{1}}\cdots x_{i_{s}})^{\frac{1}{\binom{r-1}{s-1}}}\bigg)^{\binom{r}{s}} \notag \\
&\leq  r\sum_{e\in E(H)} \frac{1}{\binom{r}{s}}\sum_{\{i_{1},\ldots,i_{s}\}\subset e}(x_{i_{1}}\cdots x_{i_{s}})^{\frac{\binom{r}{s}}{\binom{r-1}{s-1}}}  \notag \\
&= \frac{s}{\binom{r-1}{s-1}}\sum_{e\in E(H)} \sum_{\{i_{1},\ldots,i_{s}\}\subset e}(x_{i_{1}})^{\frac{r}{s}}\cdots(x_{i_{s}})^{\frac{r}{s}} \notag \\
&= \frac{s}{\binom{r-1}{s-1}} \sum_{\{i_{1},\ldots,i_{s}\}\in E(\partial_{s} H)}(x_{i_{1}})^{\frac{r}{s}}\cdots(x_{i_{s}})^{\frac{r}{s}}. \label{gj2}
\end{align}
Putting (\ref{gj1}) and (\ref{gj2}) together, we find that
\begin{align*} \notag
q(H)&\leq \frac{1}{\binom{r-1}{s-1}}\bigg(\sum_{i\in V(\partial_{s}H)}d_{\partial_{s}H}(i)(x_{i}^{r/s})^{s}+
s\sum_{\{i_{1},\ldots,i_{s}\}\in E(\partial_{s} H)}(x_{i_{1}})^{\frac{r}{s}}\cdots(x_{i_{s}})^{\frac{r}{s}}\bigg)\\ \notag 
&\leq \frac{1}{\binom{r-1}{s-1}}q(\partial_{s} H),
\end{align*}
where the last inequality follows from $\sum_{i\in V(\partial_{s} H)}(x_{i}^{r/s})^{s}=1$. 
This completes the proof. 
\end{proof}

The following lemma has its root in the recent works 
of Gao, Chang and Hou \cite{GCH2022} 
as well as She et al. \cite{SFKH2022}. 
Since the proof employs a similar line, we omit the details for brevity.  

\begin{lem} \label{lem-return}
Let $k\ge 2,\ell\ge 1$ be fixed integers and $n$ be sufficiently large. 
    If $H$ is a linear $r$-graph and $\partial H$ contains a copy of $K_{k+1}[t]$ with $t=\Omega (\log n)$, then $H$ contains an $r$-expansion of $K_{k+1}[\ell]$. 
\end{lem}

With the aid of Lemmas \ref{gj} and \ref{lem-return}, 
we now prove Theorem \ref{lhslES}.

\begin{proof}[{\bf Proof of Theorem \ref{lhslES}}]
Suppose that $H$ is an $r$-uniform hypergraph with $q(H)\ge \frac{2n}{r-1}(1- \frac{1}{k} + \varepsilon)$. 
 By setting $s=2$ in Lemma \ref{gj}, we have
\[ q(\partial H)\geq (r-1)q(H)\geq 
\left(1-\frac{1}{k}+ \varepsilon \right)2n. \] 
Let $n$ be sufficiently large.  Theorem \ref{lsES} implies that $\partial H$ contains a copy of $K_{k+1}[t]$ with $t=\Omega_{k,\varepsilon}(\log n)$. 
By Lemma \ref{lem-return}, we see that $H$ contains an $r$-expansion of $K_{k+1}[\ell]$. 
\end{proof}

To show Theorem \ref{ESSL}, we need to introduce the notion from design theory. 

\begin{definition} \label{defn-8-3}
Let $m,k,r\geq2$ be integers. An $(m,k,r)$-design is a triple $(V,X,\mathcal{B})$ satisfying 
\begin{itemize}
\item[(a)] $V$ is a set of $mk$ elements (called points),
\item[(b)] $X=\{V_{1},\ldots,V_{k}\}$ is a partition of $V$, where every $V_i$ is an $m$-subset (called groups),
\item[(c)] $\mathcal{B}$ is a family of subsets of $V$ of size $r$ (called blocks),

\item[(d)] 
Each block of $\mathcal{B}$ intersects any group of $X$ in at most one point,

\item[(e)] Each $2$-set of $V$ from two distinct groups of $X$ is contained in exactly one block of $\mathcal{B}$.
\end{itemize}
\end{definition}

By the definition, we see that every point of $V$ is contained in exactly $\frac{m(k-1)}{r-1}$ blocks of $\mathcal{B}$, and any two blocks intersect in at most one point. 
So an $(m,k,r)$-design can be regarded as a linear $\frac{m(k-1)}{r-1}$-regular $r$-graph of order $mk$. Furthermore, an $(m,k,r)$-design also is a special group divisible $2$-design of order $mk$. In 2011, Moh\'acsy \cite{Moh2011} characterized the existence of such a design.

\begin{lem}[See \cite{Moh2011}]\label{design}
Let $k$, $r$ be positive integers with $k\geq r\geq2$. Then there exists $m_{0}$ such that for any integer $m\geq m_{0}$ there exists an
$(m,k,r)$-design satisfying the condition:
\begin{equation}\label{mod}
 m(k-1) \equiv 0 \hspace{1em} \mathrm{mod}~(r-1), \quad 
 m^{2}k(k-1) \equiv 0 \hspace{1em} \mathrm{mod}~~r(r-1).
 \end{equation}
\end{lem}

Next, we are ready to prove Theorem \ref{ESSL}. 

\begin{proof}[{\bf Proof of Theorem \ref{ESSL}}] 
Let $F$ be a $2$-graph with chromatic number $\chi(F) =k+1$ and $k\ge r\ge 2$.  
 Let $m,k$ and $r$ be positive integers that satisfy (\ref{mod}) in Lemma \ref{design}.  Then for sufficiently large $m$,  there exists a linear $\frac{m(k-1)}{r-1}$-regular $r$-graph $H$ of order $mk$. 
 Since $\chi (F)=k+1$, by (d) in Definition \ref{defn-8-3}, 
 we know that $H$ does not contain a copy of $F^r$.   
 We denote $n :=km$. 
 It follows that $q(H)=\frac{2m(k-1)}{r-1}=\frac{2n}{r-1}(1-\frac{1}{k})$. Thus, for every sufficiently large $n$, we have 
\begin{equation*}
\mathrm{ex}_{q}^{lin}(n,F^{r}) \geq \frac{2n}{r-1}\bigg(1-\frac{1}{k}\bigg) + o(n).
\end{equation*} 
For the upper bound on $\mathrm{ex}_{q}^{lin}(n,F^{r})$, 
we will apply Theorem \ref{lhslES}. 
Since $\chi(F)=k+1$, there exists an integer $\ell$ such that $F$ is a subgraph of $K_{k+1}[\ell]$. 
By Theorem \ref{lhslES},
we have
$$\mathrm{ex}_{q}^{lin}(n,F^{r})\leq \mathrm{ex}_{q}^{lin}(n,(K_{k+1}[\ell])^{r})\leq \frac{2n}{r-1}\bigg(1-\frac{1}{k}+o(1)\bigg).$$
Hence, it follows that $\mathrm{ex}_{q}^{lin}(n,F^{r}) =  \frac{2n}{r-1}(1- \frac{1}{k})+ o(n)$, as needed.  
\end{proof}

\begin{theorem}\label{linear hypergraph}
Let $F$ be a color-critical graph with $\chi(F)=k+1$ and $F^{r}$ be the $r$-expansion of $F$, where $k\geq3$ and $r\geq2$. Then there exists $n_{0}$  such that for  any  $F^{r}$-free linear $r$-graph $G$ on $n\geq n_{0}$ vertices,
$$q(G)\leq \frac{2n}{r-1} \left( 1- \frac{1}{k}\right).$$
Moreover, the inequality can be achieved if $k\geq 2$, $n=mk$, where $m$, $k$ and $r$ satisfy (\ref{mod}).
\end{theorem}

\begin{proof}
Without loss of generality, we may assume that $F-e$ is a $k$-partite graph with all part sizes less than $\ell$.
 It follows that $K_{k}^{+}[\ell]$ contains a copy of $F$. Let $H$ be an $F^{r}$-free linear $r$-graph. 
 Suppose on the contrary that  $q(H)>\frac{2n}{r-1}(1- \frac{1}{k})$. Then Lemma  \ref{gj} yields 
$$q(\partial H)\geq(r-1)q(H)> \left(1- \frac{1}{k}\right)2n \geq q( T_{n,k}).$$
In view of Theorem \ref{thm-h3}, for sufficiently large $n$, the $2$-graph $\partial H$ contains a copy of $K_{k}^{+}[t]$ with $t=\Omega_k(\log n)$.
Similar to Lemma \ref{lem-return}, one can show that  $H$ contains an $r$-expansion of $K_{k}^{+}[\ell]$. This implies that $H$ contains an $r$-expansion of $F$, which is a contradiction.

Let $m,k$ and $r$ be positive integers that satisfy (\ref{mod}) in Lemma \ref{design}. We denote 
 $n=km$.  
Then for sufficiently large $n$,
 we know from Lemma \ref{design} that there exists an  $F^{r}$-free  linear $\frac{n(k-1)}{k(r-1)}$-regular $r$-graph $H$. Clearly, 
 we have $q(H)=\frac{2n(k-1)}{k(r-1)}$, which completes the proof.
\end{proof}

\section{Concluding remarks}

\label{sec-9}

In this paper, we have systematically investigated the spectral extremal graph problems for the signless Laplacian spectral radius. 
For example, we have studied the supersaturation and blowup phenomenon for cliques (Theorems \ref{thm-1-2} and \ref{lsES}), and the sufficient condition to guarantee a large copy of color-critical graphs, joints and books (Theorems \ref{thm-h3}, \ref{thm-q-js} and \ref{thm-genelized-books}). Moreover, we have proved the stability result for a general graph in terms of the signless Laplacian spectral radius (Theorem \ref{sst}). 
As an application, we can obtain the upper bound on the power of degrees of a graph that does not contain a general graph as a substructure (Theorem \ref{thm-2-14} 
and Corollary \ref{coro-1-7}).

We summarize the aforementioned results in the following table. 

\begin{table}[H]
\centering
\begin{tabular}{ccccccccccc}
\toprule
  & $e(G)$ & $\lambda (G)$ & $q(G)$ \\
\midrule
Tur\'{a}n theorem & \cite{Turan41} 
 & \cite{Gui1996,Niki2007laa2} &  \cite{HJZ2013}  \\
 Erd\H{o}s-Stone-Simonovits & \cite{ES1946,ES1966}  & \cite{Gui1996,ESB2009} & \cite{ZLS2025}  \\ 
 Supersaturation & \cite{ES1983-super}  
 &  \cite{BN2007jctb} &  Current paper \\
 Blowups & \cite{BE1973} & \cite{ESB2009} & Current paper \\ 
 Color-critical graphs & \cite{S1968}  & \cite{N2009} & Current paper \\ 
 Joints & \cite{Erd1969,BN2008} & \cite{N2009,LLZ2024-book-4-cycle} & Current paper \\  
 Generalized books & \cite{LLZ2024-book-4-cycle}  
 & \cite{LLZ2024-book-4-cycle} & Current paper \\  
  Stability result & \cite{Erd1966Sta1,Erd1966Sta2,S1968} &  \cite{Niki2009jgt} & Current paper \\ 
\bottomrule 
\end{tabular}
\caption{Some classical results and its spectral correspondence.}
 \label{tab-SSSR}
\end{table}

Our results generalized the classical extremal graph results in terms of the size or the adjacency spectral radius. 
In the end of this paper, we conclude some interesting problems for the readers.

\subsection{Counting color-critical graphs}

In 2023, Ning and Zhai \cite{NZ2021} proved that if $G$ is an $n$-vertex graph with $\lambda (G)\ge \lambda (T_{n,2})$, then $G$ contains at least $\lfloor \frac{n}{2}\rfloor -1$ triangles, unless $G=T_{n,2}$. 
In the concluding remarks of \cite{NZ2021}, Ning and Zhai wrote that
for the case of triangles, is there some interesting
phenomenon when we consider the relationship between the number of triangles
and signless Laplacian spectral radius, Laplacian spectral radius, distance spectral
radius and etc? 
As noted by Li, Lu and Peng \cite[Sec. 5]{LLP2024-AAM}, 
the counting result of triangles does not hold 
in terms of the signless Laplacian spectral radius. 
Indeed, let $K_{n-1,1}^+$ be the graph obtained by adding an edge to the independent set of the star $K_{n-1,1}$. 
We can see that $q(K_{n-1,1}^+) > q(K_{n-1,1}) =q(T_{n,2})$, but $K_{n-1,1}^+$ contains exactly one triangle.

It seems possible to establish the spectral counting result for cliques $K_{k+1}$ in the case $k \ge 3$. 
The classical supersaturation for cliques was  understood by the result of Erd\H{o}s \cite{Erd1969}. 
We refer the interested readers to \cite{LFP2024-triangular,LFP-count-bowtie,Mub2010,PY2017} for more related results. 

\begin{theorem}[Erd\H{o}s \cite{Erd1969}] \label{thm-count-cliques}
     Let $k\ge 2,t\ge 1$ be fixed integers and $n$ be sufficiently large. 
 If $G$ is an $n$-vertex graph with 
$e(G)\ge e(T_{n,k})+t $, then  
$G$ contains at least 
$t(\frac{n}{k})^{k-1} +O(n^{k-2})$ copies of $K_{k+1}$. 
\end{theorem}

Motivated by Theorem \ref{thm-count-cliques}, 
 we propose the following conjecture.

\begin{conj} \label{conj-1}
    Let $k\ge 3$ be fixed and $n$ be sufficiently large. 
    If $G$ is an $n$-vertex graph with 
    \[ q(G)> \left(1- \frac{1}{k} \right)2n, \]
    then $G$ contains at least 
    $(\frac{n}{k})^{k-1} + O(n^{k-2})$ copies of $K_{k+1}$. 
\end{conj}

We note that the supersaturation in Conjecture \ref{conj-1} differs from that in Theorem \ref{thm-1-2}, which says that every graph $G$ with $q(G)> (1- \frac{1}{k})2n + \varepsilon n$ contains $\Omega_{k,\varepsilon} (n^{k+1})$ copies of $K_{k+1}$. 

Let $F$ be a color-critical graph with $\chi (F)=k+1\ge 4$. Theorem \ref{thm-h3} implies that if $n$ is sufficiently large and $q(G)> (1- \frac{1}{k})2n$, then $G$ contains a copy of $F$. 
More generally, it is interesting to investigate the spectral result for counting the color-critical graph $F$, rather than the clique $K_{k+1}$.

\subsection{Connections between the spectral radius and the size}

Under what conditions does an edge-extremal problem admit a spectral counterpart, and more specifically, when do both versions share identical extremal graphs? 
For example, when we forbid a clique \cite{Niki2007laa2} or a color-critical graph \cite{N2009}, it is known that the spectral extremal graphs are the same as the classical edge-extremal graphs. 
Confirming the Cioab\u{a}--Desai--Tait conjecture \cite{CDT21}, 
Wang, Kang and Xue \cite{Wang2022} proved the following interesting result, which says roughly that the spectral extremal graph of $F$ must be an edge-extremal graph when $F$ is `almost' color-critical. 
Recall that $\mathrm{Ex}(n,F)$ denotes the set of all $n$-vertex $F$-free graphs
with maximum number of edges.  

\begin{theorem}[Wang--Kang--Xue \cite{Wang2022}] \label{thm-WKX}
Let $k\ge 2$ be an integer and $F$ be a graph with $\mathrm{ex}(n,F)= e(T_{n,k}) + O(1)$. For sufficiently large $n$, if $G$ is an $n$-vertex graph with the maximal adjacency spectral radius
over all $n$-vertex $F$-free graphs, then
$ G\in \mathrm{Ex}(n,F)$. 
\end{theorem}

{Byrne \cite{Byr2026} determined the optimal value $c_1(k)$ for all $k\ge 3$, such that if $n$ is sufficiently large and $\mathcal{F}$ is a graph family satisfying $\mathrm{ex}(n,\mathcal{F}) \le e(T_{n,k}) + Qn $, where $Q < c_1(k)$, then Theorem \ref{thm-WKX} remains valid. 
 We refer to \cite{FLZ2025,FTZ2024} for further connections between the spectral Tur\'{a}n problem and the edge Tur\'{a}n problem.}    
 The proofs of these results are based on the spectral stability Theorem \ref{thm-Nik-stability}. 
 In this paper, we have established the spectral stability for the signless Laplacian spectral radius; see Theorem \ref{sst}. 
{Inspired by these results, we propose the following problem for readers.

\begin{problem} \label{conj-9-4}
   Characterize all graphs $F$ such that for sufficiently large $n$, 
    if $G$ is an $n$-vertex graph that achieves the maximal $q(G)$ among all $n$-vertex ${F}$-free graphs, then 
    $G\in \mathrm{Ex}(n,{F})$. 
\end{problem}
 } 

Problem \ref{conj-9-4} was solved for cliques \cite{HJZ2013} and color-critical graphs  \cite{ZLL2025}. {A natural step toward Problem \ref{conj-9-4} is to extend the recent results \cite{Wang2022,FLZ2025,Byr2026} to the $Q$-spectral setting.} 
In addition, it is feasible to apply our stability result to determine the $Q$-spectral extremal graphs for some classical non-bipartite graphs, such as, the intersecting cliques, disjoint cliques, and the edge blow-up of graphs.

\section*{Acknowledgements}
The authors are grateful to the anonymous reviewers for their insightful feedback on this work.
We would also like to thank Dr. Xiaocong He for reading an earlier draft of this paper. 
Jian Zheng was partially supported by the National Natural Science Foundation of China (No. 12161047) and Jiangxi Provincial Natural Science foundation (No. 20224BCD41001). 
Yongtao Li (Corresponding author) was supported by the Postdoctoral Fellowship Program of CPSF (No. GZC20233196). Yi-Zheng Fan was 
supported by the National Natural Science Foundation of China (No. 12331012).


\section*{Declaration of competing interest}
The authors declare that they have no conflicts of interest to this work.

\section*{Data availability}
No data was used for the research described in the article.

\appendix

\section{Proof of Theorem \ref{h1-two-cases}} 
\label{sec-A}

\begin{proof}[{\bf Proof of Theorem \ref{h1-two-cases}}]
 Let $n$ be sufficiently large.
We define a sequence of graphs $G_{0},G_1,\ldots, G_{k}$ by the following procedure Algorithm \ref{alg:graph_pruning}. Here, the graph $G_i$ has $n-i$ vertices. 

\begin{algorithm}[H] 
\caption{Graph Reduction Algorithm}
\label{alg:graph_pruning}
\begin{algorithmic}[1]
    \STATE Set $G_0 = G$;
    \STATE Set $i = 0$;
    \WHILE{$\bm{\delta(G_i) \leq (\gamma - \alpha)(n - i)}$ \textbf{and} $\bm{i < \lfloor \beta n \rfloor}$} 
        \STATE Select a nonnegative unit eigenvector $(x_1, \ldots, x_{n-i})$ corresponding to $q(G_i)$;
        \STATE Select a vertex $u_i \in V(G_i)$ such that $x_{u_i} = \min\{x_1, \ldots, x_{n-i}\}$;
        \STATE Set $G_{i+1} = G_i - u_i$;
        \STATE Add $1$ to $i$;
    \ENDWHILE 
\end{algorithmic}
\end{algorithm} 
Note that $|G_k|=n-k\geq n- \lfloor\beta n\rfloor\geq(1-\beta)n$. 
We will show that
\begin{equation}\label{e1}
q(G_k)>2\gamma\bigg(1+\frac{k\alpha}{7n}\bigg)|G_k|.
\end{equation} 
To do so, we will prove by induction that 
for every $i=0,1, \ldots,k$, 
\begin{equation}\label{e2}
\frac{q(G_{i})}{n-i}\geq\bigg(1+\frac{i\alpha}{6n}\bigg)\frac{q(G)}{n}.
\end{equation}
The assertion is trivial for $i=0$. Let $0\leq i\leq k-1$ and assume that (\ref{e2}) holds for $i$. Now we consider the graph $G_{i+1}$.  
We denote $\delta=\delta(G_{i})$ and $q=q(G_{i})$. Note  that 
\begin{equation}\label{e3}
\delta\leq (\gamma-\alpha)(n-i),
\end{equation}
and
\begin{equation}\label{e4}
q\geq (n-i)\bigg(1+\frac{i\alpha}{6n}\bigg)\frac{q(G)}{n}\geq (n-i)
\left(2\gamma- \frac{s}{n^{2}}\right)\geq 
\frac{5(n-i)}{4},
\end{equation} 
where the second (resp. third) inequality follows from $q(G)>2\gamma n- s/n$ (resp. $\gamma\geq31/48)$. 

 Recall that $x_{u_i}$ is the minimum entry of $(x_1,\ldots ,x_{n-i})$ of $q(G_i)$.  
 It follows by Lemma \ref{min} that
 $$x_{u_{i}}^{2}\leq \frac{\delta}{q^{2}-2q\delta+(n-i)\delta}=\frac{\delta}{(q-\delta)^{2}-\delta^{2}+(n-i)\delta}.$$
Note that the right-hand side increases with $\delta$ and decreases with $q$ on $[\delta,+\infty)$.
In view of (\ref{e3}) and (\ref{e4}), we get that
\begin{displaymath}
\begin{split}
x_{{u}_{i}}^{2}(n-i)
&\leq \frac{(\gamma-\alpha)(n-i)^{2}}{(\gamma+\alpha-s/n^{2})^{2}(n-i)^{2}
-(\gamma-\alpha)^{2}(n-i)^{2}+(\gamma-\alpha)(n-i)^{2}}\\
&\leq \frac{(\gamma-\alpha)}{(\gamma+\alpha)^{2}-(\gamma-\alpha)^{2}+\gamma-\alpha
-2s(\gamma+\alpha)/n^{2}}\\
&\leq \frac{\gamma -\alpha}{\gamma-\alpha
+3\gamma \alpha}
= 1- \frac{3\gamma \alpha}{\gamma-\alpha
+3\gamma \alpha}\\
&\leq 1- \frac{3\gamma \alpha}{\gamma
+3\gamma \alpha}\\
&\leq  1-\alpha,
\end{split}
\end{displaymath}
where the last inequality follows from $0<\alpha\leq2/3$. 
Recall that $G_i$ has $n-i$ vertices. 
Combining this with  (\ref{e4}) and Lemma \ref{sdf}, we obtain that   
\begin{align*} \notag
\frac{q(G_{i+1})}{n-i-2}&\ge \frac{q(G_i)}{n-i-1}\bigg(1+\frac{1-(n-i)x_u^{2}}{(n-i-2)(1-x_u^{2})}\bigg)-\frac{1-(n-i)x_u^{2}}{(n-i-2)(1-x_u^{2})}\\ 
\notag 
&\geq \frac{q(G_i)}{n-i-1}\bigg(1+\frac{1-(n-i)x_u^{2}}{5(n-i-2)(1-x_u^{2})}\bigg)+\frac{5}{4}\cdot \frac{4(1-(n-i)x_u^{2})}{5(n-i-2)(1-x_u^{2})}-\frac{1-(n-i)x_u^{2}}{(n-i-2)(1-x_u^{2})}\\ 
\notag 
&\geq \frac{q(G_i)}{n-i-1}\bigg(1+\frac{1-(n-i)x_u^{2}}{5n}\bigg).
\end{align*} 
Then 
\begin{displaymath}
\begin{split}
\frac{q(G_{i+1})}{n-i-1}&\geq\frac{q(G_{i})}{n-i}\bigg(1+\frac{1}{n-i-1}\bigg)\bigg(1-\frac{1}{n-i-1}\bigg)\bigg(1+\frac{1-(n-i)x_{u}^{2}}{5n}\bigg)\\
&\geq \frac{q(G)}{n}\bigg(1+\frac{i \alpha}{6n}\bigg)\bigg(1-\frac{1}{(n-i-1)^{2}}\bigg)\bigg(1+\frac{\alpha}{5n}\bigg)\\
&\geq \frac{q(G)}{n}\bigg(1+\frac{i \alpha}{6n}\bigg)\bigg(1+\frac{\alpha}{6n}\bigg)\\
&\geq \frac{q(G)}{n}\bigg(1+\frac{(i+1)\alpha}{6n}\bigg),
 \end{split}
\end{displaymath} 
which completes the induction step and the proof of (\ref{e2}).

 By inequality (\ref{e2}) and $q(G)> 2\gamma n - s/n$, we have
 $$\frac{q(G_{k})}{n-k}\geq   \bigg(1+\frac{k\alpha}{6n}\bigg) \frac{q(G)}{n} >  
 \bigg(1+\frac{k\alpha}{6n}\bigg)\bigg(2\gamma-\frac{s}{n^{2}}\bigg)
 >2\gamma\bigg(1+\frac{k\alpha}{7n}\bigg),$$
which completes the proof of (\ref{e1}).

By the construction of the sequence of graphs in Algorithm \ref{alg:graph_pruning}, 
we know that when the procedure $\mathcal{P}$ stops at some $G_k$, 
we have either $k=\lfloor\beta n\rfloor$ 
or $\delta(G_k)> (\gamma-\alpha)|G_k|$. We denote $H:=G_k$. 
Next, we show that $H$ is the desired graph satisfying the 
statement (i) or (ii). 

If $k=\lfloor\beta n\rfloor$, then applying (\ref{e1}) gives 
$$q(H)>2\gamma\bigg(1+\frac{\lfloor\beta n\rfloor\alpha}{7n}\bigg)|H|>2\gamma\bigg(1+\frac{\alpha\beta}{8}\bigg)|H|.$$
Therefore, $H$ satisfies the statement (i).

If $k<\lfloor\beta n\rfloor$, then $\delta(H)> (\gamma-\alpha)|H|$. In view of (\ref{e1}), we find that
$$q(H)>2\gamma\bigg(1+\frac{k\alpha}{7n}\bigg)|H|>2\gamma |H|.$$
Hence, the statement (ii) holds.
\end{proof}


\begin{thebibliography}{199}
\bibitem{Bollobas78}
B. Bollob\'as, Extremal Graph Theory, 
Academic Press, New York, 1978.

\bibitem{BE1973} 
B. Bollob\'as, P.  Erd\H{o}s, 
On the structure of edge graphs,  
{Bull. London Math. Soc.}, 5 (1973) 317--321. 


\bibitem{BES1976}
B. Bollob\'{a}s, P. Erd\H{o}s, M. Simonovits, 
On the structure of edge graphs II,
J. London Math. Soc., 12 (1976) 219--224.


\bibitem{BK1994}
B. Bollob\'{a}s,  Y. Kohayakawa, 
An extension of the Erd\H{o}s--Stone theorem, 
Combinatorica 14 (3) (1994) 279--286. 

\bibitem{BN2004}
B. Bollob\'{a}s, V. Nikiforov,
Degree powers in graphs with forbidden subgraphs,
 Electron.  J. Combin. 11 (2004), R42.

 
 \bibitem{BN2005}
B. Bollob\'{a}s, V. Nikiforov, 
Books in graphs, 
European J. Combin. 26 (2005) 259--270.  


\bibitem{BN2007jctb} 
B. Bollob\'{a}s, V. Nikiforov, 
Cliques and the spectral radius, 
 J. Combin. Theory Ser.  B 97 (2007) 859--865. 

 \bibitem{BN2008} 
B. Bollob\'{a}s, V. Nikiforov, 
Joints in graphs, 
Discrete Math. 308 (2008) 9--19. 

\bibitem{BN2011} 
B. Bollob\'{a}s, V. Nikiforov, 
Large joints in graphs,  
European J. Combin. 32 (1) (2011) 33--44.  

\bibitem{BN2012}
B. Bollob\'{a}s, V. Nikiforov,
Degree powers in graphs: the Erd\H{o}s--Stone theorem,
Combin. Probab. Comput. 21 (2012) 89--105.

\bibitem{BLWX2024} 
D. Brada\v{c}, H. Liu, Z. Wu, Z. Xu,  Clique density vs blowups, Combin. Probab. Comput. (2026), 1–22. \url{ https://doi.org/10.1017/S0963548325100333}

\bibitem{BL2024}
G. Brooks, W. Linz, 
Some exact and asymptotic results for hypergraph
Tur\'{a}n problems in $\ell_2$-norm, 
(2024), arXiv:2310.09379. 

 
\bibitem{Byr2026}
J. Byrne, A sharp spectral extremal result for general non-bipartite graphs, 
Linear Algebra Appl. 733 (2026) 75–115.


 \bibitem{CY2000}
Y. Caro, R. Yuster, A Tur\'{a}n type problem concerning the powers of the degrees of a
graph, Electron. J. Combin. 7 (2000), R47.

\bibitem{CLZ2020} 
M.-Z, Chen, A.-M. Liu, X.-D. Zhang, 
The signless Laplacian spectral radius of graphs with forbidding linear forests,  
{Linear Algebra Appl.}, 591 (2020) 25--43.

\bibitem{CLZ2022}
M.-Z. Chen, Z.-M. Li, X.-D. Zhang, 
The signless Laplacian spectral radius of graphs without trees, (2022), arXiv:2209.03120. 

\bibitem{CZ} 
M.-Z. Chen, A.-M. Liu, X.-D. Zhang,
The signless Laplacian spectral radius of graphs without
intersecting odd cycles,
Electron. J. Linear Algebra 40 (2024) 370--381. 

\bibitem{CJZ2025}
M.-Z. Chen, Y.-L. Jin, P.-L. Zhang, 
The signless Laplacian spectral radius of book-free graphs, 
Discrete Math. 348 (2025), No. 114448. 

\bibitem{CWZ2022}
W. Chen, B. Wang, M. Zhai, 
Signless Laplacian spectral radius of graphs without short cycles or long cycles, 
Linear Algebra Appl. 645 (2022) 123--136.

\bibitem{CS1981}
V. Chv\'{a}tal, E. Szemer\'{e}di,
On the Erd\H{o}s--Stone theorem, 
J. London Math. Soc., (2) 23 (1981) 207--214. 


 \bibitem{CDT21}
 S. Cioab\u{a}, D.N. Desai, M. Tait, 
The spectral radius of graphs with no odd wheels, 
European J. Combin. 99 (2022), No.103420. 

\bibitem{AN2013} 
N.M.M. de Abreu, V. Nikiforov, Maxima of the $Q$-index: Graphs with bounded clique number,
 {Electronic J. Linear Algebra}, 24 (2013) 121--130.

 \bibitem{FNP2013} 
M.A.A. de Freitas, V. Nikiforov, L. Patuzzi, Maxima of the Q-index: Forbidden 4-cycle and 5-cycle, {Electron. J. Linear Algebra.}, 26 (2013) 905--916.

\bibitem{FNP2016} 
M.A.A. de Freitas, V. Nikiforov, L. Patuzzi, Maxima of the Q-index: graphs with no $K_{s,t}$, 
{Linear Algebra Appl.}, 496 (2016) 381--391.


\bibitem{DKLNTW2021}
D.N. Desai, L. Kang, Y. Li, Z. Ni, M. Tait, J. Wang, 
Spectral extremal graphs for intersecting cliques, 
Linear Algebra Appl. 644 (2022) 234--258.  
 
 
\bibitem{Erd1962a}
P. Erd\H{o}s,
On a theorem of Rademacher--Tur\'{a}n,
Illinois J. Math. 6 (1962) 122--127.  

\bibitem{Erd1966Sta1}
P. Erd\H{o}s, Some recent results on extremal problems in graph theory (Results), In: Theory of Graphs (International Symposium Rome, 1966), Gordon and Breach, New York, Dunod, Paris, 1966, pp. 117--123.

\bibitem{Erd1966Sta2} 
P. Erd\H{o}s, On some new inequalities concerning extremal properties of graphs, In: Theory of Graphs (Proceedings of the Colloquium, Tihany, 1966), Academic Press, New York, 1968, pp. 77--81.  

\bibitem{Erd1969}
P. Erd\H{o}s, 
On the number of complete subgraphs and circuits contained in graphs, 
Casopis P\v{e}st. Mat. 94 (1969) 290--296. 

 \bibitem{EFR1992}
 P. Erd\H{o}s, R. Faudree, C. Rousseau, 
 Extremal problems involving vertices and 
 edges on odd cycles, 
Discrete Math. 101 (1) (1992) 23--31.

\bibitem{ES1966}
P. Erd\H{o}s, M. Simonovits, 
A limit theorem in graph theory, 
Stud. Sci. Math. Hungar. 1 (1966) 51--57.  

\bibitem{ES1983-super} 
P. Erd\H{o}s, M. Simonovits, 
Supersaturated graphs and hypergraphs, 
Combinatorica 3 (1983) 181--192.  

\bibitem{ES1946} 
P.  Erd\H{o}s, A.H. Stone, 
On the structure of linear graphs,  
{Bull. Amer. Math. Soc.}, 52 (1946) 1087--1091.

\bibitem{FLZ2025}
L. Fang, H. Lin, M. Zhai, 
The spectral Tur\'{a}n problem: Characterizing
spectral-consistent graphs, 17 pages, 
(2025), arXiv:2508.12070.  

\bibitem{FTZ2024}
L. Fang, M. Tait, M. Zhai, 
Decomposition family and spectral extremal problems on non-bipartite graphs, 
Discrete Math. 348 (2025), No. 114527.


\bibitem{FW2023}
J. Fox, Y. Wigderson, 
Minimum degree and the graph removal lemma, 
J. Graph Theory 102 (2023) 648--665.

\bibitem{FHW2021}
 J. Fox, X. He, Y. Wigderson, 
 Ramsey goodness of books revisited, 
Advances in Combin. 4 (2023),  21 pages, 
arXiv:2109.09205.  

\bibitem{FS13} 
Z. F\"uredi,  M. Simonovits,
The history of degenerate (bipartite) extremal graph problems,  
in Erd\H{o}s Centennial, 
Bolyai Soc. Math. Stud., 25, 
J\'{a}nos Bolyai Math. Soc., Budapest, 2013, 
pp. 169--264. 

\bibitem{GC2021}
G. Gao, A. Chang, 
A linear hypergraph extension of the bipartite Tur\'an problem, 
European J. Combin., 93 (2021), No. 103269.
 
\bibitem{GCH2022} 
G. Gao, A. Chang, Y. Hou, Spectral radius on  linear $r$-graphs without expanded $K_{r+1}$, 
SIAM J. Discrete Math., 36 (2) (2022) 1000--1011.

\bibitem{G2025}
D. Gerbner, On degree powers and counting stars in $F$-free graphs, 
European J. Combin. 126 (2025) 104--135.

\bibitem{Ger2025}
D. Gerbner, 
On non-degenerate Tur\'{a}n problems for
expansions, 
European J. Combin. 124 (2025), No. 104071.

\bibitem{Gui1996}
B.D. Guiduli, 
 Spectral extrema for graphs, 
 Ph.D. dissertation, 
University of Chicago, 1996. (unpublished)
See \url{http://people.cs.uchicago.edu/~laci/students/guiduli-phd.pdf}  


\bibitem{HJZ2013} 
B. He, Y.-L. Jin, X.-D. Zhang, Sharp bounds for the signless Laplacian spectral radius in terms of clique number,  
{Linear Algebra Appl.}, 438 (2013) 3851--3861.

\bibitem{HCC2021} 
Y. Hou, A. Chang, J. Cooper, Spectral extremal results for hypergraphs,  
{Electron. J. Combin.}, 28 (3) (2021), P3.46. 
 
\bibitem{Ish2002}
Y. Ishigami, 
Proof of a conjecture of Bollob\'{a}s and Kohayakawa on 
the Erd\H{o}s--Stone theorem, 
J. Combin. Theory Ser. B 85 (2002) 222--254.

\bibitem{Kee2011}
P. Keevash,  Hypergraph Tur\'{a}n problems, 
Surveys in combinatorics, London Math. Soc. Lecture Note Ser., 392, Cambridge Univ. Press, Cambridge, 2011, pp. 83--139.

\bibitem{KLM2014} P. Keevash, J. Lenz, D. Mubayi, Spectral extremal problems for hypergraphs,  
{SIAM J. Discrete Math.}, 28 (4) (2014) 1838--1854.

\bibitem{KN1979}
N. Khad\v{z}iivanov, V. Nikiforov, 
Solution of a problem of {P. Erd\H{o}s} 
about the maximum number of
  triangles with a common edge in a graph, 
C. R. Acad. Bulgare Sci. 32 (1979) 1315--1318. 

\bibitem{LLQ2019} Y. Lan, H. Liu, Z. Qin, Y. Shi, Degree powers in graphs with a forbidden forest, Discrete Math. 342(2019) 821--835.


\bibitem{LLP2018}
Y. Li, Y. Liu, X. Peng, 
Signless Laplacian spectral radius and Hamiltonicity of graphs with large minimum degree,  
Linear Multilinear Algebra 66 (2018) (10) 2011--2023. 
 

\bibitem{LiLF} 
Y. Li, W. Liu, L. Feng, A survey on spectral conditions for some extremal graph problems,  
{Adv. Math. (China)}, 51 (2022) 193--258.

\bibitem{LLP2024-AAM}
Y. Li, L. Lu, Y. Peng, 
A spectral Erd\H{o}s--Rademacher theorem, 
Adv. in Appl. Math. 158 (2024), No. 102720.


\bibitem{LFP2024-triangular}
Y. Li, L. Feng, Y. Peng, 
A spectral Erd\H{o}s--Faudree--Rousseau theorem, 
J. Graph Theory 110 (4) (2025) 408--425.  


\bibitem{LFP-count-bowtie}
Y. Li, L. Feng, Y. Peng, 
Spectral supersaturation: Triangles and bowties, 
European J. Combin.  128 (2025), No. 104171. 

\bibitem{LLZ2024-book-4-cycle}
Y. Li, H. Liu, S. Zhang, 
More on Nosal's spectral theorem: Books and 4-cycles, 
23 pages, (2025), arXiv:2508.14366. 


\bibitem{LLZ2025-stability}
Y. Li, H. Liu, S. Zhang, 
An edge-spectral Erd\H{o}s--Stone--Simonovits theorem and its stability, 
30 pages, (2025), arXiv:2508.15271. 

\bibitem{LLZ2025+}
Y. Li, H. Liu, S. Zhang, 
Edge-spectral Tur\'{a}n theorems for color-critical graphs with applications, 
26 pages, (2025), arXiv:2511.15431.  

\bibitem{LN2016}
B. Li, B. Ning, 
Spectral analogues of Erd\H{o}s' and Moon--Moser's theorems on Hamilton cycles, 
Linear Multilinear Algebra 64 (11) (2016) 2252--2269.

\bibitem{LN2023}
L. Liu, B. Ning, 
Spectral Tur\'{a}n-type problems on sparse spanning graphs, Discrete Math. 349 (2026), No. 115016.  

\bibitem{LL2009-Zagreb}
M. Liu, B. Liu, New sharp upper bounds for the first Zagreb index, MATCH Commun. Math. Comput. Chem., 62 (3) (2009) 689--698.

\bibitem{LMX2022}
R. Liu, L. Miao, J. Xue, 
Maxima of the $Q$-index of non-bipartite $C_3$-free graphs, 
Linear Algebra Appl. 673 (2023) 1--13. 




\bibitem{LGW2021}
Z. Lou, J.-M. Guo, Z. Wang, 
Maxima of $L$-index and $Q$-index: Graphs with given size and diameter, 
Discrete Math. 344 (2021), No. 112533.

\bibitem{MLX2022}
L. Miao, R. Liu, J. Xue, 
Maxima of the $Q$-index of non-bipartite graphs: 
forbidden short odd cycles, 
Discrete Appl. Math. 340 (2023) 104--114.   

\bibitem{Moh2011} H. Moh\'acsy, The asymptotic existence of group divisble designs of large order with index one,
 {J. Combin. Theory Ser A}, 118 (7) (2011) 1915--1924.

 \bibitem{Mub2010}
 D. Mubayi, 
 Counting substructures I: Color critical graphs, 
 Adv. Math. 225 (5) (2010) 2731--2740.

 \bibitem{NWK2022}
Z. Ni, J. Wang, L. Kang, 
Spectral extremal graphs for disjoint cliques, 
Electronic J. Combin.  30 (1) (2023), \#P1.20. 


\bibitem{Niki2007laa2} 
V. Nikiforov, 
Bounds on graph eigenvalues II, 
Linear Algebra Appl. 427 (2007) 183--189.  

\bibitem{Niki2008blms} 
V. Nikiforov, 
Graphs with many $r$-cliques have large complete $r$-partite subgraphs, 
Bull. London Math. Soc. 40 (2008) 23--25. 

\bibitem{Nik2008-LAA}
V. Nikiforov, 
A spectral condition for odd cycles in graphs, 
Linear Algebra Appl. 428 (2008) 1492--1498.

\bibitem{Niki2009-degree}
V. Nikiforov,
Degree powers in graphs with a forbidden even cycle,
Electron. J. Combin. 15 (2009), R107.

\bibitem{ESB2009} V. Nikiforov, A spectral  Erd\H{o}s--Stone--Bollob\'as theorem,  
{Combin. Probab. Comput.}, 18 (3) (2009) 455--458.

\bibitem{N2009} V. Nikiforov, Spectral saturation: Inverting the spectral Tur\'an theorem,  {Electron. J. Combin.},
16 (1) (2009), R33.


\bibitem{Niki2009jgt} 
V. Nikiforov, 
Stability for large forbidden subgraphs,  
J. Graph Theory 62  (4) (2009)  362--368.  

\bibitem{N2010} V. Nikiforov, Tur\'an Theorem inverted,  
{Discrete Math.}, 310 (2010) 125--131.


\bibitem{Niki2011}
V. Nikiforov, 
Some new results in extremal graph theory, 
in Surveys in Combinatorics, 
Cambridge University Press, 2011, pp. 141--181.

 \bibitem{Niki2021}
 V. Nikiforov,  On a theorem of Nosal, 
 12 pages (2021), arXiv:2104.12171. 

 \bibitem{NR2004}
V. Nikiforov, C.C. Rousseau, Large generalized books are $p$-good,
J. Combin. Theory Ser B 92 (2004) 85--97.

\bibitem{NY1} V. Nikiforov, X. Yuan, 
Maxima of the $Q$-index: graphs without long path,  
{Electron. J. Linear Algebra}, 27 (2014) 504--514.

\bibitem{NY2} 
V. Nikiforov, X. Yuan, 
Maxima of the $Q$-index: forbidden even cycles,  
{Linear Algebra Appl.}, 471 (2015) 636-653.


  \bibitem{NZ2021}
B. Ning, M. Zhai, Counting substructures and eigenvalues I: Triangles,
European J. Combin.  110 (2023), No. 103685.

 \bibitem{PY2017}
O. Pikhurko, Z. Yilma, 
Supersaturation problem for color-critical graphs, 
J. Combin. Theory Ser. B 123 (2017) 148--185.


\bibitem{SFK2025}
C.-M. She, Y.-Z. Fan, L. Kang,
Spectral bipartite Tur\'{a}n problems on linear hypergraphs, 
Discrete Math. 348 (2025), No. 114435. 

\bibitem{SFKH2022} 
C.-M. She, Y.-Z. Fan, L. Kang, Y. Hou, 
Linear spectral Tur\'{a}n problems for expansions of graphs with given chromatic number, Acta Math. Appl. Sin. Engl. Ser. (2025). \url{https://doi.org/10.1007/s10255-024-1156-x} 

\bibitem{S1968} M. Simonovits, A method for solving extremal problems in graph theory, stability problems,  {In Theory of Graphs (Proc. Colloq., Tihany, 1966)}, pp.279-319, Academic Press, New York, 1968. 

\bibitem{Sim13}
M. Simonovits,  
Paul Erd\H{o}s' influence on Extremal graph theory, 
in The Mathematics of Paul Erd\H{o}s II, 
R.L. Graham, Springer, New York, 2013, pp. 245--311. 

\bibitem{Turan41}
P. Tur\'{a}n, 
On an extremal problem in graph theory, 
Mat. Fiz. Lapok 48 (1941), pp. 436--452. 
(in Hungarian).

\bibitem{Wang2022}
J. Wang, L. Kang, Y. Xue, 
On a conjecture of spectral extremal problems, 
J. Combin. Theory Ser. B 159 (2023) 20--41.  

\bibitem{WZ2023} 
B. Wang, M. Zhai, 
Maxima of the $Q$-index: forbidden a fan,  
{Discrete Math.}, 34 (2023), No. 113264.

\bibitem{Wil1986} 
H. Wilf, 
Spectral bounds for the clique and indendence numbers of graphs, 
J. Combin. Theory Ser. B 40 (1986) 113--117.  


\bibitem{Yu2008}
G. Yu, On the maximal signless Laplacian spectral radius of graphs with given matching number, Proc. Jpn. Acad., Ser. A, Math. Sci. 84 (2008) 163--166.

 \bibitem{Y2014} 
 X. Yuan, Maxima of the $Q$-index: forbidden odd cycles,  {Linear Algebra Appl.}, 458 (2014) 207--216.

 \bibitem{ZL2022jgt}
M. Zhai, H. Lin,
A strengthening of the spectral color critical edge theorem: Books and theta graphs,
J. Graph Theory 102 (3) (2023) 502--520. 


 \bibitem{ZLS2021}
 M. Zhai,  H. Lin,  J. Shu, 
Spectral extrema of graphs with fixed size: Cycles and complete bipartite graphs, 
European J. Combin. 95 (2021), No. 103322. 

\bibitem{ZXL2022}
M. Zhai, J. Xue, R. Liu, 
An extremal problem on $Q$-spectral radii of graphs with given size and matching number, 
Linear Multilinear Algebra 70 (20) (2022) 
5334--5345.  

 \bibitem{ZXL2020}
M. Zhai, J. Xue, Z. Lou,  
The signless Laplacian spectral radius of graphs with a prescribed number of edges, 
Linear Algebra Appl. 603 (2020) 154--165.


\bibitem{Zhang2022}
L. Zhang, Degree powers in $K_{s,t}$-minor free graphs, Discrete Math., 345 (4) (2022), No. 112783. 


\bibitem{ZHG2021} 
Y. Zhao, X. Huang, H. Guo, The signless Laplacian spectral radius of graphs with no intersecting triangles,  
{Linear Algebra Appl.}, 618 (2021) 12--21.

\bibitem{ZLL2025}
J. Zheng, Y. Li, H. Li, 
The signless Laplacian spectral Tur\'{a}n problems for color-critical graphs, 
Linear Algebra Appl. 730 (2026) 546--565. 

\bibitem{ZLS2025} 
J. Zheng, H. Li, L. Su, 
A signless Laplacian spectral Erd\H{o}s--Stone--Simonovits theorem, 
Discrete Math. 349 (2026), No. 114665. 

\bibitem{Zhou2010}
B. Zhou,
Signless Laplacian spectral radius and Hamiltonicity,
Linear Algebra Appl. 432 (2010) 566--570.

\bibitem{ZWL2020}
Q. Zhou, L. Wang, Y. Lu, 
Signless Laplacian spectral conditions for Hamilton-connected graphs with large minimum degree, 
Linear Algebra Appl. 592 (2020) 48--64.



\end{thebibliography}
\end{document}